\documentclass[11pt]{article}

\usepackage[utf8]{inputenc}
\usepackage{scalerel}
\usepackage{enumerate}
\usepackage{stmaryrd}
\usepackage{authblk}
\usepackage{wrapfig}
\usepackage{tikz-cd}
\usepackage{amsmath,amscd,amssymb}
\usepackage{amsthm}
\usepackage{amsfonts}
\usepackage{mathtools}
\usepackage{color}
\usepackage[mathscr]{eucal}
\usepackage{amssymb}
\usepackage[margin=1in]{geometry}
\usepackage[]{lineno}
\usepackage{hyperref}

\hypersetup{colorlinks,allcolors=black}
 
\graphicspath{{figures/}}

 \newtheorem{theorem}{Theorem}[section]
 \newtheorem{proposition}[theorem]{Proposition}
 \newtheorem{corollary}[theorem]{Corollary}
 \newtheorem{lemma}[theorem]{Lemma}

 \theoremstyle{definition}
 \newtheorem{definition}[theorem]{Definition}

\usepackage{chngcntr}
\usepackage{apptools}
\AtAppendix{\counterwithin{theorem}{section}}
\newcommand{\real}{\ensuremath{\mathbb{R}}}
\newcommand{\lip}{\ensuremath{\text{Lip}}}
\newcommand{\obs}[1]{\ensuremath{O(#1)}}
\newcommand{\obsp}[1]{\ensuremath{O_p(#1)}}
\newcommand{\cobs}[1]{\ensuremath{O^c_{#1}}}

\newcommand{\pcobs}[1]{\ensuremath{Q^c_{#1}}}
\newcommand{\pcobsp}[1]{\ensuremath{Q^c_{p,#1}}}
\newcommand{\borel}{\ensuremath{\mathcal{B}}}

\newcommand{\mean}[1]{\ensuremath{M_\mathcal{#1}}}
\newcommand{\meanp}[1]{\ensuremath{M_{#1,p}}}
\newcommand{\dks}{\ensuremath{d_{KS,p}}}
\newcommand{\dgh}{\ensuremath{d_{GH}}}

\newcommand{\dlp}[1]{\ensuremath{r_{p,#1}}}
\newcommand{\dlpc}[1]{\ensuremath{c_{p,#1}}}
\newcommand{\covp}[1]{\ensuremath{\Sigma_{p,#1}}}

\newcommand{\mm}[1]{\ensuremath{\mathcal{#1}}}

\begin{document}

\title{Observable Covariance and Principal Observable Analysis \\ for Data on Metric Spaces}
\author[1]{Ece Karacam}
\author[1]{Washington Mio}
\author[2]{Osman Berat Okutan}
\affil[1]{Department of Mathematics, Florida State University}
\affil[2]{Max Planck Institute for Mathematics in the Sciences}
\date{ }

\maketitle

\begin{abstract}
Datasets consisting of objects such as shapes, networks, images, or signals overlaid on such geometric objects permeate data science. Such datasets are often equipped with metrics that quantify the similarity or divergence between any pair of elements turning them into metric spaces $(X,d)$, or a metric measure space $(X,d,\mu)$ if data density is also accounted for through a probability measure $\mu$. This paper develops a Lipschitz geometry approach to analysis of metric measure spaces based on metric observables; that is, 1-Lipschitz scalar fields $f \colon X \to \real$ that provide reductions of $(X,d,\mu)$ to $\real$ through the projected measure $f_\sharp (\mu)$. Collectively, metric observables capture a wealth of information about the shape of $(X,d,\mu)$ at all spatial scales. In particular, we can define stable statistics such as the {\em observable mean} and {\em observable covariance} operators $M_\mu$ and $\Sigma_\mu$, respectively. Through a maximization of variance principle, analogous to principal component analysis, $\Sigma_\mu$ leads to an approach to vectorization, dimension reduction, and visualization of metric measure data that we term {\em principal observable analysis}. The method also yields basis functions for representation of signals on $X$ in the {\em observable domain}.


\end{abstract}

\medskip

\section{Introduction}

Datasets comprising objects that can be represented as points in a metric space $(X,d)$ are of interest in many domains of application. The nature of such objects can vary broadly: point clouds or images with similarity measured via the Wasserstein or Gromov-Wasserstein distance (e.g., \cite{earthmovers98,memoli2011}), collections of shapes with geodesic distances in shape space as exemplified by \cite{kendall1984,younes2000,mm2006,liu2010,charon2013}, and persistence diagrams with the bottleneck distance in topological data analysis (cf.\,\cite{carlsson2009,Chazal2016}) are just some among myriad examples. Data density can also be critically important, in which case $(X,d)$ is further equipped with a (Borel) probability measure $\mu$. Such a triple $(X,d, \mu)$ is known as a {\em metric measure space}, abbreviated $mm$-space. This paper presents a method for probing and analyzing the structure of $mm$-spaces that is rooted in principles reminiscent of those of classical mechanics, where the state of a complex system is probed via observables (measurements). In that setting, a general observable is a continuous mapping $f \colon X \to \real$, where $X$ is the state space of the system. However, as we are interested in the shape of $(X,d, \mu)$ and general observables can markedly distort the shape of both $(X,d)$ and $\mu$, we impose the additional requirement that observables be 1-Lipschitz functions; that is,
\begin{equation}
|f(x)-f(y)| \leq d(x,y),
\end{equation}
for any $x,y \in X$. We refer to these as {\em metric observables} or, relaxing terminology, simply {\em observables}. The 1-Lipschitz condition has the virtue of being uniform across all spatial scales; that is, independent of how small or large $d(x,y)$ may be. Although a single observable only serves as a weak shape descriptor, the empirical evidence suggests that in the aggregate they retain rich information about the shape of $(X,d,\mu)$.  

We define the concepts of {\em observable mean} and {\em observable covariance} that provide a pathway to statistical analysis in the metric measure setting. The observable mean functional simply associates to each observable its expected value and the covariance tensor associates to each pair of observables the correlation of the pair. This allows for the development of an analogue of principal component analysis (PCA) for metric measure data that we term {\em principal observable analysis} (POA). The first principal observable (PO1) $\phi_1 \colon X \to \real$ is the centered (i.e., mean-zero) observable that maximizes the variance of the measure $\mu$ projected to $\real$ via $\phi_1$; that is, the variance of the push-forward measure ${\phi_1}_\sharp (\mu)$.  To define higher principal observables, we iterate the process by maximizing the variance among the centered observables uncorrelated ($\mu$-orthogonal) to those already  constructed. Among other things, POA gives a method for vectorizing $mm$-data, as well as a technique for dimension reduction, visualization, and statistical analysis. Although beyond the scope of this paper, we note that POA vectorization also provides a means for inputting metric data into machine learning pipelines such as neural networks.

\begin{figure}[ht]
\begin{center}
\begin{tabular}{cc}
\begin{tabular}{c}
\includegraphics[width=0.22\linewidth]{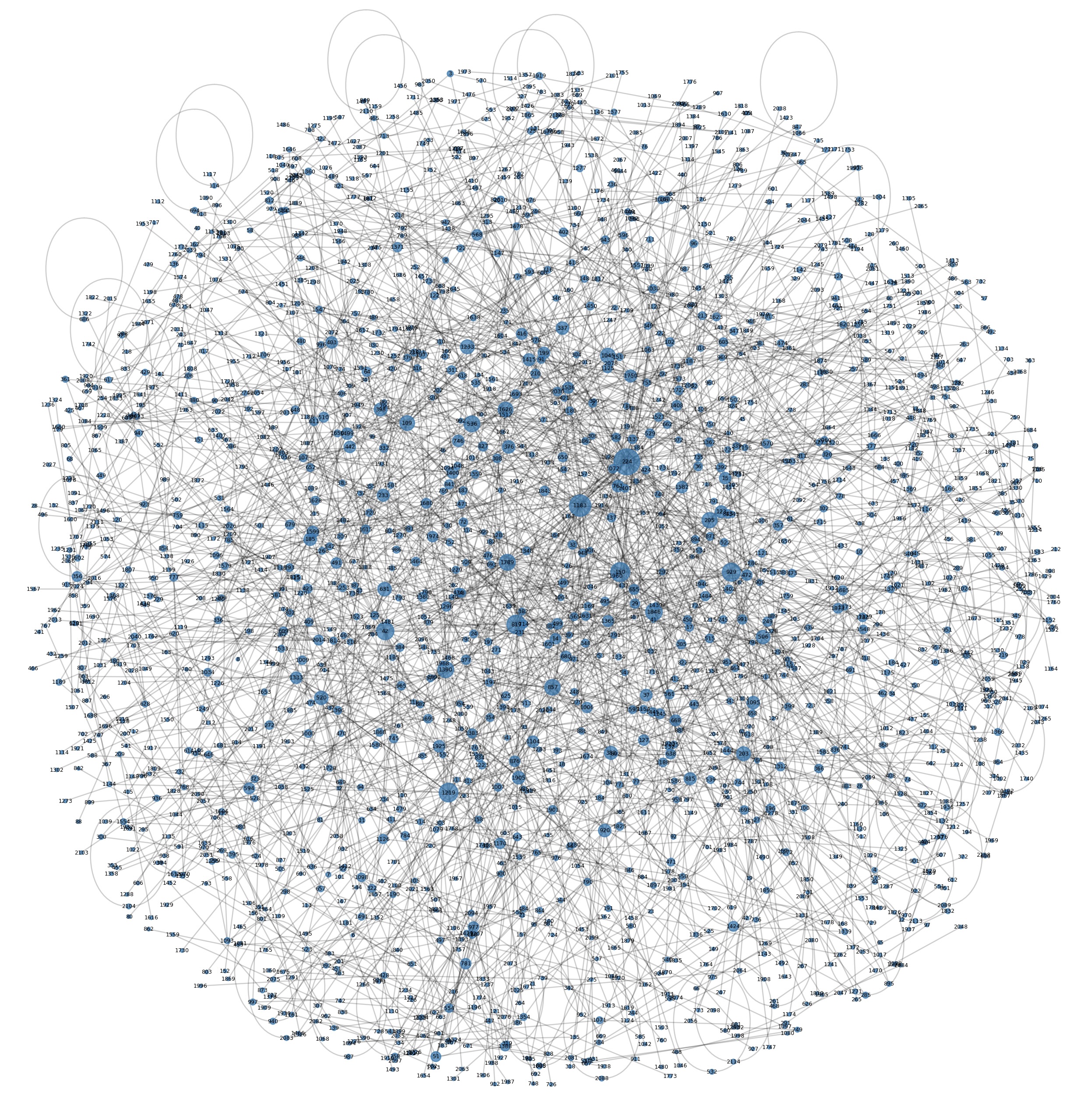} 
\end{tabular}
\qquad & \qquad
\begin{tabular}{c} 
\includegraphics[width=0.25\linewidth]{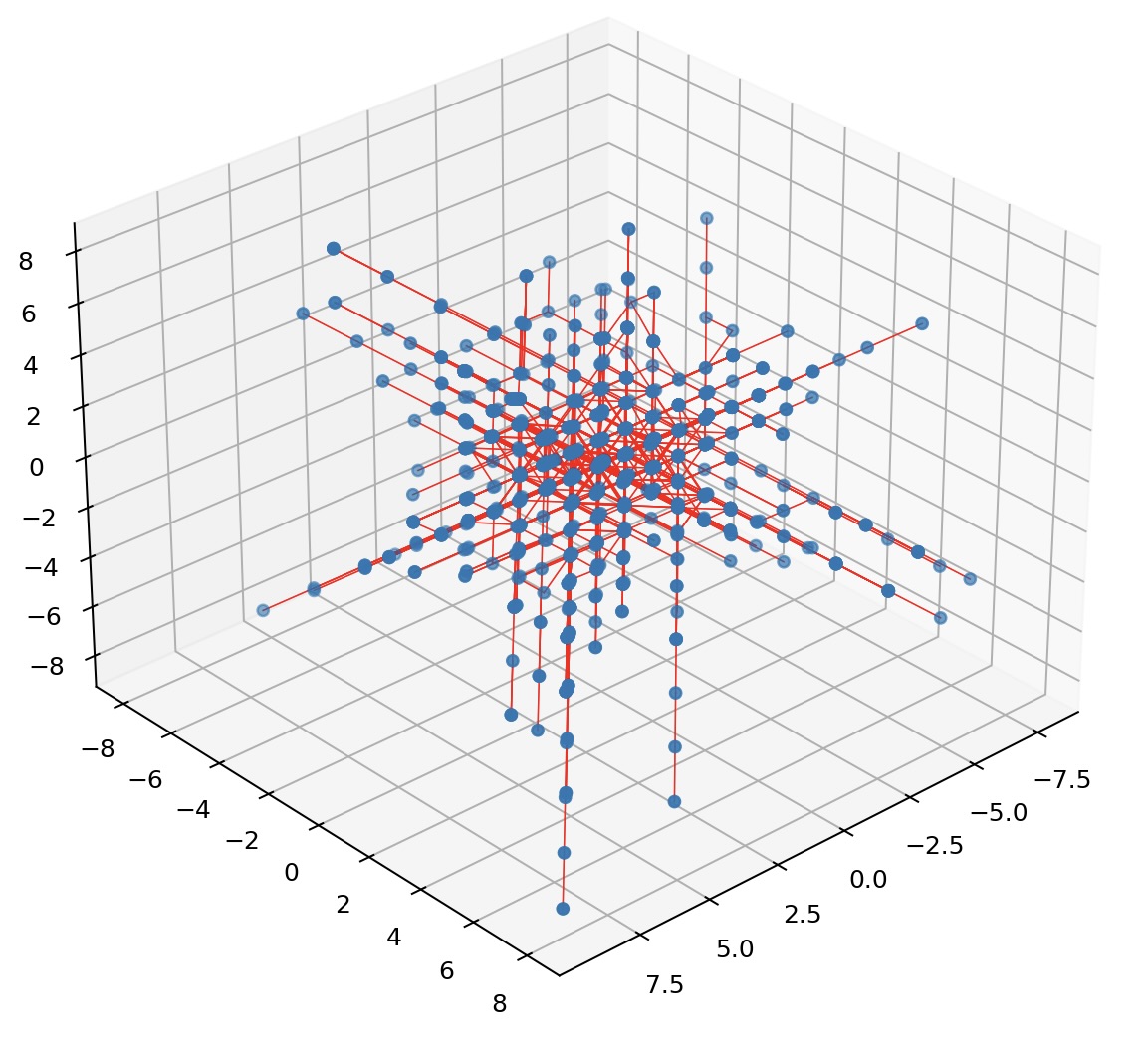} 
\end{tabular} \\
(a) A protein-protein interaction network \qquad & \qquad (b) POA of the largest component \vspace{0.1in}\\
\begin{tabular}{c}
\includegraphics[width=0.35\linewidth]{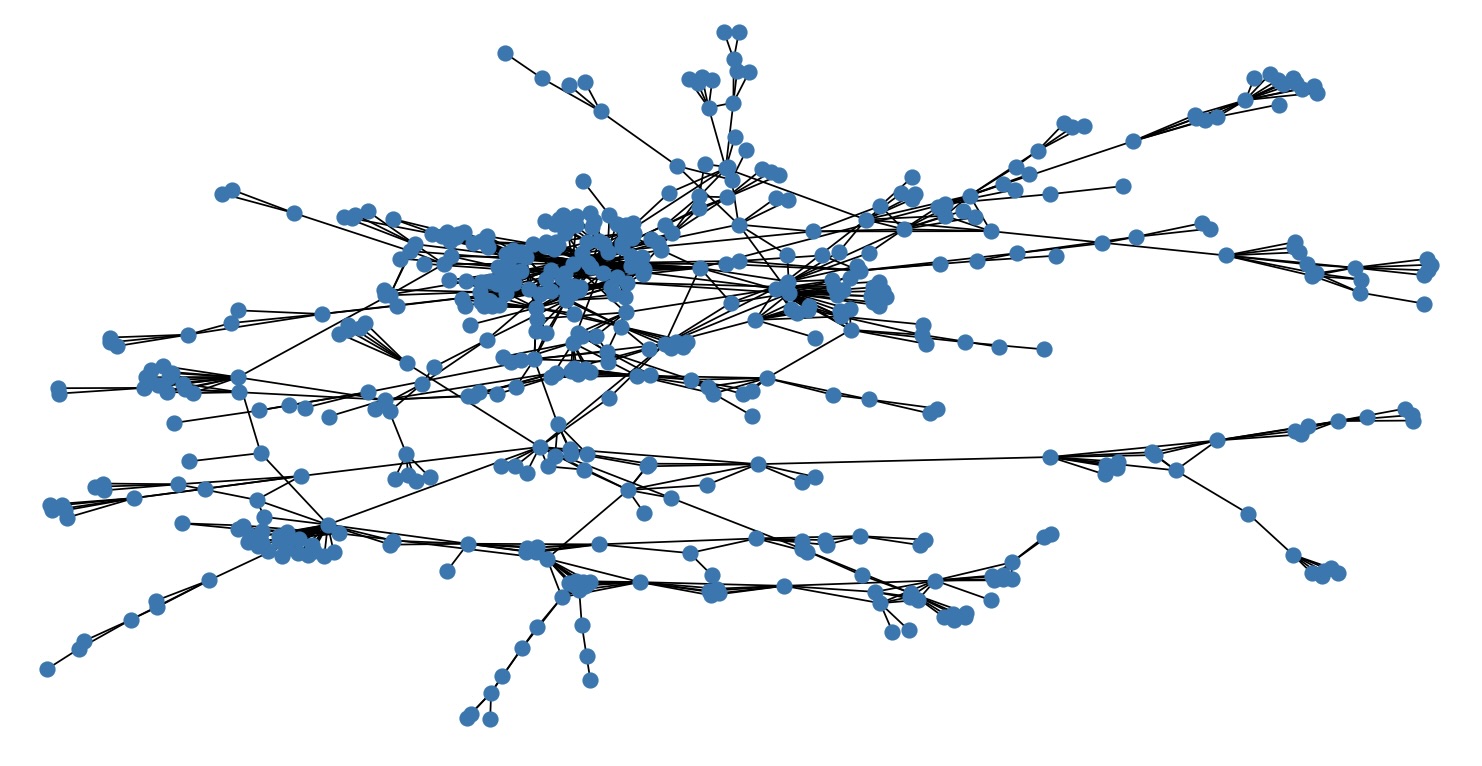} 
\end{tabular} 
\qquad & \qquad
\begin{tabular}{c} 
\includegraphics[width=0.25\linewidth]{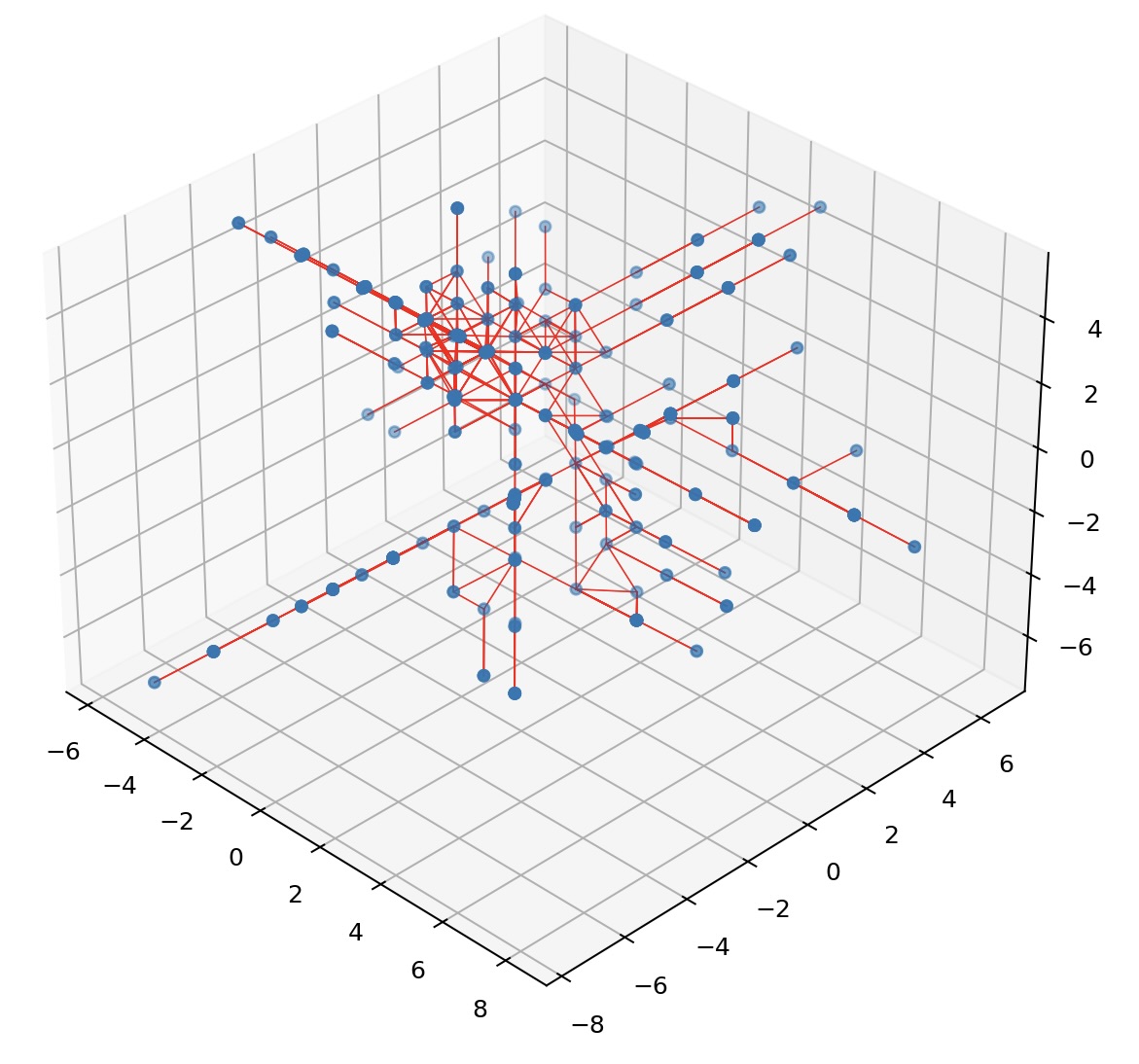} 
\end{tabular} \\
(c) A diseasome network \qquad & \qquad (d) POA representation
\end{tabular} 
\caption{(a) the protein-protein interaction network for yeast {\em(Saccharomyces cerevisiae)}; (b) the 3-dimensional POA representation of its largest connected component; (c) a diseasome network and (d) its 3-dimensional POA representation.}
\label{F:3dpoa}
\end{center}
\end{figure}

To illustrate the method, Figure \ref{F:3dpoa}(a) shows a protein-protein interaction network for yeast {\em (Saccharomyces cerevisiae)}, data obtained from the Network Data Repository \cite{NDR2015}. The node set $V$ of the largest connected component of the network represents 1458 proteins and is viewed as a metric space with the shortest path distance (each edge is assigned unitary length) and the probability measure on $V$ attributes equal weights to all nodes. Figure \ref{F:3dpoa}(b) depicts a 3-dimensional representation of $V$, obtained through the first 3 principal observables and drawn with the same network structure to facilitate visualization. Figures \ref{F:3dpoa}(c) and \ref{F:3dpoa}(d) show a similar example. The original graph is a (connected) diseasome network with 516 nodes representing different diseases. The edges indicate shared genetic origins, molecular mechanisms, or other biological relationships (data also from Network Data Repository).

Much in the spirit of Fourier transforms on networks or compact Riemannian manifolds (cf.\,\cite{rosenberg_1997,hammond2009}), principal observables also yield basis functions for encoding and analyzing signals $f \colon X \to \real$ in the {\em observable domain}. This is an additional application of principal observables that can be further explored and for which we provide some empirical evidence. We also remark that, although POA and PCA are built on the shared principle of maximization of variance of 1-dimensional ``projections'', unlike principal components, the number of principal observables is infinite if the cardinality of $X$ is infinite regardless of the (topological) dimension of $X$. 

We note that metric observables have been used in the study of $mm$-spaces by Gromov to define an {\em observable distance} between probability measures on (possibly distinct) metric domains \cite{gromov2007mg}. In contrast, this paper employs observables to produce a statistical summary for any distribution on a (compact) metric space and as a dimension reduction, vectorization and visualization tool. A special type of metric observables is given by distance-to-a-point functions $f_b \colon X \to \real$, $b \in X$, given by $f_b (x) = d(x,b)$. Distance distributions associated with such observables have been employed by M\'{e}moli and Needham to address various inverse problems for $mm$-spaces in \cite{distdistr22}, which also provides references to additional related literature. 

As stability and consistency are crucial for practical usefulness of statistical summaries of a distribution, we address these problems for the observable mean and observable covariance. For distributions on a fixed metric space, stability is proven with respect to the Wasserstein distance. As there is an increasing interest in analysis and fusion of heterogeneous data, we also investigate stability for distributions defined on different metric domains using the $L_p$ transportation distances between $mm$-spaces due to Sturm \cite{sturm2006}. Stability of principal observables is a more delicate question that requires further investigation. Even for standard PCA, advances in this direction by Pennec are relatively recent, involving a formulation based on flags of principal subspaces \cite{pennec2018}.

\paragraph{Organization.} Section \ref{S:obs} introduces metric observables and the observable mean and covariance. Principal observable analysis is developed in Section \ref{S:poa}, including illustrations of POA applied to metric data. Section \ref{S:basis} is devoted to principal observables viewed as basis functions for signal analysis. Section \ref{S:proof} proves stability theorems for the observable mean and covariance on a fixed metric space, whereas Section \ref{S:heterogeneous} addresses stability for heterogeneous data. Section \ref{S:summary} presents a summary and some discussion. 

%---------------------------

\section{Observable Mean and Covariance} \label{S:obs}

Throughout the paper $(X,d)$ denotes a compact metric space. A scalar field $f \colon X \to \real$ is $K$-Lipschitz, $K \geq 0$, if 
\begin{equation}
|f(x) - f(y)| \leq K d(x,y),
\label{E:lip}
\end{equation}
for any $x,y \in X$. This imposes some control on the rate at which $f$ can distort distances, uniformly across all spatial scales; that is, regardless of how close or far apart the points $x$ and $y$ may be. We denote by $\lip (X,d)$ the vector space of all Lipschitz functions on $(X,d)$ equipped with the $\sup$ norm
\begin{equation}
\|f\|_\infty = \sup_{x \in X} |f(x)| \,.
\end{equation}
$\lip (X,d)$ is an infinite dimensional function space if $|X| = \infty$. We refer to 1-Lipschitz (non-expanding) functions as {\em metric observables} (or simply {\em observables}) on $X$ and denote by $\obs{X} \subseteq \lip (X,d)$ the closed, convex  subset of all observables on $(X,d)$.

The collection of all Borel probability measures on $(X,d)$ is denoted $\borel (X,d)$. For $\mu \in \borel (X,d)$, define the {\em mean functional} $M_\mu \colon \lip (X,d) \to \real$ by
\begin{equation}
M_\mu (f) := \mathbb{E}_\mu[f] = \int_X f(x) \, d\mu (x),
\end{equation}
which assigns to an observable $f$ its expected value. The {\em observable mean} is the restriction of $M_\mu$ to $\obs {X}$. Abusing notation, we denote it $M_\mu \colon \obs{X} \to \real$. We also adopt the abbreviation $M_\mu f$ for $M_\mu(f)$.

To define the covariance tensor, let $\lip^c_\mu \subseteq \lip (X,d)$ be the subspace of $\mu$-centered Lipschitz functions given by
\begin{equation}
\lip^c_\mu  = \{f \in \lip (X,d) \colon M_\mu f=0\}
\end{equation}
and $\cobs{\mu} := \obs{X} \cap \lip^c_\mu$, the closed and convex subspace of $\mu$-centered observables. The {\em covariance tensor} of $(X,d,\mu)$ is the bilinear form $\Sigma_\mu \colon \lip^c_\mu \times \lip^c_\mu \to \real$ given by
\begin{equation}
(f,g) \mapsto \int_X f(x) g(x) \, d\mu (x) 
\end{equation}
and the {\em observable covariance} is the restriction $\Sigma_\mu \colon \cobs{\mu} \times \cobs{\mu} \to \real$ to the subspace of $\mu$-centered metric observables, which we also denote by $\Sigma_\mu$. The variance of a (not necessarily centered) observable $f \in \obs{X}$ is given by
\begin{equation}
\sigma^2 (f) = \Sigma_\mu (f-M_\mu f, f-M_\mu f)
\end{equation}
and the correlation between $f,g \in \obs{X}$ by
\begin{equation}
\rho (f,g) = \Sigma_\mu (f-M_\mu f,g-M_\mu g) \,.
\end{equation}

Next, we discuss the continuity of the observable mean and observable covariance, where $\obs{X}$ is equipped with the metric induced by the $\|\cdot\|_\infty$ norm and the product space $\obs{X} \times \obs{X}$ with the metric induced by the norm $\|(f,g)\|_\infty = \max \{\|f\|_\infty, \|g\|_\infty\}$.

\begin{lemma} \label{L:bound}
If $g \in \cobs{\mu}$, then $\|g\|_\infty \leq D_X$, where $D_X$ is the diameter of $(X,d)$.
\end{lemma}

\begin{proof}
If $g$ is identically zero, then the lemma is trivially satisfied. Thus, assume $g \not\equiv 0$. Since $M_\mu(g)=0$, there must be points where $g$ attains positive and negative values, and the compactness of $X$ ensures that there exist points $x^+, x^- \in X$ where $g$ attains its maximum and minimum values, respectively. Hence, $g(x^+) = \max \{g(x) \colon x \in X\} >0$ and $g(x^-) = \min \{g(x) \colon x \in X\} <0$.

We claim that $g(x^+) < D_X$ and $-D_X < g(x^-)$. Indeed, suppose $g(x^+)\geq D_X$. Then, since $g(x^-)<0$, we have
\begin{equation}
g(x^+) - g(x^-) > g(x^+) \geq D_X \geq d(x^+, x^-),
\end{equation}
contradicting the fact that $g$ is 1-Lipschitz. A similar argument shows that $g(x^-) > -D_X$. Therefore, $-D_X < g(x) < D_X$, $\forall x \in X$, proving the lemma.
\end{proof}

\begin{proposition} \label{P:continuity}
If $\mu \in \borel (X,d)$ and $(X,d)$ is compact, then the observable mean and observable covariance satisfy:
\begin{enumerate}[(i)]
\item $|M_\mu (f) - M_\mu(g)| \leq \|f-g\|_\infty$, for any $f,g \in \obs{X}$;
\item For any $(f_1,g_1), (f_2,g_2) \in \cobs{\mu} \times \cobs{\mu}$,
\[
|\Sigma_\mu (f_1,g_1) - \Sigma_\mu (f_2,g_2)|  \leq 2D_X \cdot \|(f_1,g_1)-(f_2,g_2)\|_\infty,
\]
where $D_X$ is the diameter of $(X,d)$. In particular, $|\sigma^2 (f) - \sigma^2(g)| \leq 2 D_X \cdot \|f-g\|_\infty$,
for any $f,g \in \cobs{\mu}$.
\end{enumerate}
\end{proposition}

\begin{proof}
(i) $|M_\mu (f)-M_\mu(g)| \leq \int_X |f(x)-g(x)| d\mu (x) \leq \sup_{x\in X} |f(x)-g(x)| = \|f-g\|_\infty$.

\medskip

(ii) By Lemma \ref{L:bound}, if $f,g \in \cobs{\mu}$,
\begin{equation} \label{E:cont1}
|\Sigma_\mu (f,g)| \leq \int_X |f(x) g(x)| d\mu (x) \leq \|f\|_\infty \|g\|_\infty \leq D_X \cdot \|f\|_\infty.
\end{equation}
By \eqref{E:cont1}, for any $(f_1,g_1), (f_2,g_2) \in \cobs{\mu} \times \cobs{\mu}$, we have
\begin{equation}\label{E:cont2}
\begin{split}
|\Sigma_\mu (f_1,g_1) - \Sigma_\mu (f_2,g_2)| &\leq 
|\Sigma_\mu (f_1,g_1) - \Sigma_\mu (f_2,g_1)| + |\Sigma_\mu (f_2,g_1) - \Sigma_\mu (f_2,g_2)|  \\
&= |\Sigma_\mu (f_1-f_2,g_1)| + |\Sigma_\mu (f_2,g_1-g_2)| \\
&\leq D_X \cdot \|f_1-f_2\|_\infty + D_X \cdot \|g_1-g_2\|_\infty \\
&\leq 2 D_X \cdot \max\{\|f_1-f_2\|_\infty, \|g_1-g_2\|_\infty\} \\
&= 2D_X \cdot \|(f_1,g_1)-(f_2,g_2)\|_\infty,
\end{split}
\end{equation}
as claimed. Lastly, \eqref{E:cont2} implies that
\begin{equation}
|\sigma^2 (f) - \sigma^2(g)| = |\Sigma_\mu(f,f) - \Sigma_\mu(g,g)|  
\leq 2D_X \cdot \|(f,f)-(g,g)\|_\infty = 2D_X \cdot \|f-g\|_\infty,
\end{equation}
concluding the proof.

\end{proof}

%----------------------------

\section{Principal Observable Analysis} \label{S:poa}

Using the observable mean and observable covariance, we develop an analogue of principal component analysis (PCA) for data on metric spaces that we call {\em principal observable analysis} (POA). Among other things, this gives a (metrically controlled) dimension reduction, visualization and vectorization method for $mm$-spaces.

\subsection{Principal Observables}

A {\em first principal observable} (PO1) for $(X,d,\mu)$, denoted $\phi_1$, is a centered metric observable of maximal non-zero variance. In other words,
\begin{equation}
\phi_1:= \mathop{\text{arg max}}_{f \in \cobs{\mu}} \sigma^2 (f) =  \mathop{\text{arg max}}_{f \in \cobs{\mu}} \Sigma_\mu(f,f) \,.
\end{equation}
Inductively, assuming that $\phi_1, \ldots, \phi_{n-1}$ have been constructed, define an $n$th {\em principal observable} $\phi_n$ as a centered metric observable that maximizes the variance among those $\mu$-orthogonal to the subspace spanned by $\phi_1, \ldots, \phi_{n-1}$, which means that $\phi_n$ is uncorrelated with $\phi_i$, $1 \leq i \leq n-1$. More precisely,
\begin{equation}
\phi_n:= \mathop{\text{arg max}}_{f \in \cobs{\mu}} \sigma^2 (f) \,,
\end{equation}
subject to the constraints $\int_X f(x) \phi_i (x) \, d\mu (x) =0$, for $1 \leq i \leq n-1$.

Since we are assuming that $X$ is compact, the existence of a first principal observable $\phi_1$ follows from the continuity of the observable covariance proven in Proposition \ref{P:continuity} and the compactness of $\cobs{\mu}$ with respect to $\|\cdot\|_\infty$, a consequence of the Arzel\`{a}-Ascoli Theorem (cf.\,\cite[Theorem~2.4.7]{dudley2002real}). Thus, as long as $\cobs{\mu}$ contains functions whose variance does not vanish, $\phi_1$ exists. The existence of $\phi_n$, $n>1$, also uses the fact that the subspace $V_n \subseteq \lip^c_\mu$ that is $\mu$-orthogonal to $\phi_1, \ldots, \phi_{n-1}$ is closed so that the domain of maximization $\cobs{\mu} \cap V_n$ is compact. Thus, $\phi_n$ exists if $\cobs{\mu} \cap V_n$ contains observables of non-trivial variance. The details of the existence argument are given in Appendix \ref{A:po}. Uniqueness, almost everywhere and up to sign, is a more delicate question but the experimental evidence is that it holds generically. Uniqueness may not hold, for example, if the $mm$-space $(X,d,\mu)$ admits symmetries. This is also the case for standard PCA  if the covariance matrix has eigenvalues with multiplicity larger than $1$.

In the calculation of $\phi_n$, $n \geq 1$, the cost function $f \mapsto \sigma^2 (f)$ is convex because it is the restriction to a convex domain of the quadratic form associated with the positive semi-definite bilinear form $\Sigma_\mu \colon \lip (X,d) \times \lip (X,d) \to \real$. Thus, the computation of principal observables involves a series of convex maximization problems. In our calculations, we employ the disciplined convex-concave programming algorithm described in \cite{shen2016disciplined}. The computational bottleneck for a finite metric space $X_n=\{x_1, \ldots, x_n\}$, $n$ large, is the enforcement of the 1-Lipschitz condition that involves a constraint for each pair of points in $X_n$, namely, $|\phi(x_i)-\phi(x_j)| \leq d(x_i,x_j)$, for $1 \leq i < j \leq n$. However, the computation can be performed efficiently if the metric space is the node set of a sparse weighted graph and the metric is the shortest path distance. As shown in the next proposition, in this case, the Lipschitz condition only needs to be verified for pairs of vertices connected by an edge, lowering the computational cost significantly.

Let $G=(V,E)$ be a simple graph with vertex set $V = \{v_1, \ldots, v_n\}$ and edge set $E$, with edge weights $w_{ij}>0$ for $\{v_i,v_j\} \in E$. Denote by $d \colon V \times V \to \real$ the metric given by the shortest path distance in $G$.

\begin{proposition}
Let $f \colon V \to \real$. If $|f(v_i)-f(v_j)| \leq w_{ij}$, for any $\{v_i,v_j\} \in E$, then $f$ is 1-Lipschitz.
\end{proposition}

\begin{proof}
Let $v,w \in V$, $v \ne w$, and $v=v_0 , v_1, \ldots, v_{k-1}, v_k = w$ be an edge path realizing the distance $d(v,w)$. This means that $\{v_{i-1}, v_i\} \in E$, for each $1\leq i \leq k$, and $d(v,w) = \sum_{i=1}^k d_{i-1,i}$. Since $d(v_{i-1}, v_i) = w_{i-1,i}$, by the triangle inequality, we have 
\begin{equation}
|f(w)-f(v)| \leq \sum_{i=1}^k |f(v_i)-f(v_{i-1})| \leq \sum_{i=1}^k w_{i-1,i} = \sum_{i=1}^k d(v_{i-1},v_i) = d(v,w),
\end{equation}
as claimed.
\end{proof}

\subsection{POA Embeddings and Dimension Reduction}

Let $\phi_1, \ldots, \phi_k$ be the first $k$ principal observables of a metric-measure space $(X,d,\mu)$. Define the $k$-dimensional POA ``embedding'' $\imath_k \colon X \to \real^k$ by $\imath_k(x) = (\phi_1 (x), \ldots, \phi_k(x))$. POA embeddings provide a vectorization technique for metric data, as well as a dimension reduction tool. Unlike standard PCA, however, the natural metric on $\real^k$ to consider is not the Euclidean ($L_2$) metric, but rather the $L_\infty$ distance given by
\begin{equation}
\|a-b\|_\infty = \max_{1\leq i\leq k} |a_i-b_i|.
\end{equation}
This is because the defining property of metric observables is that they are 1-Lipschitz and the choice of $L_\infty$ ensures that the POA embedding $\imath_k$ is also 1-Lipschitz.

\begin{figure}[ht]
\begin{center}
\begin{tabular}{cc}
\begin{tabular}{c}
\includegraphics[width=0.22\linewidth]{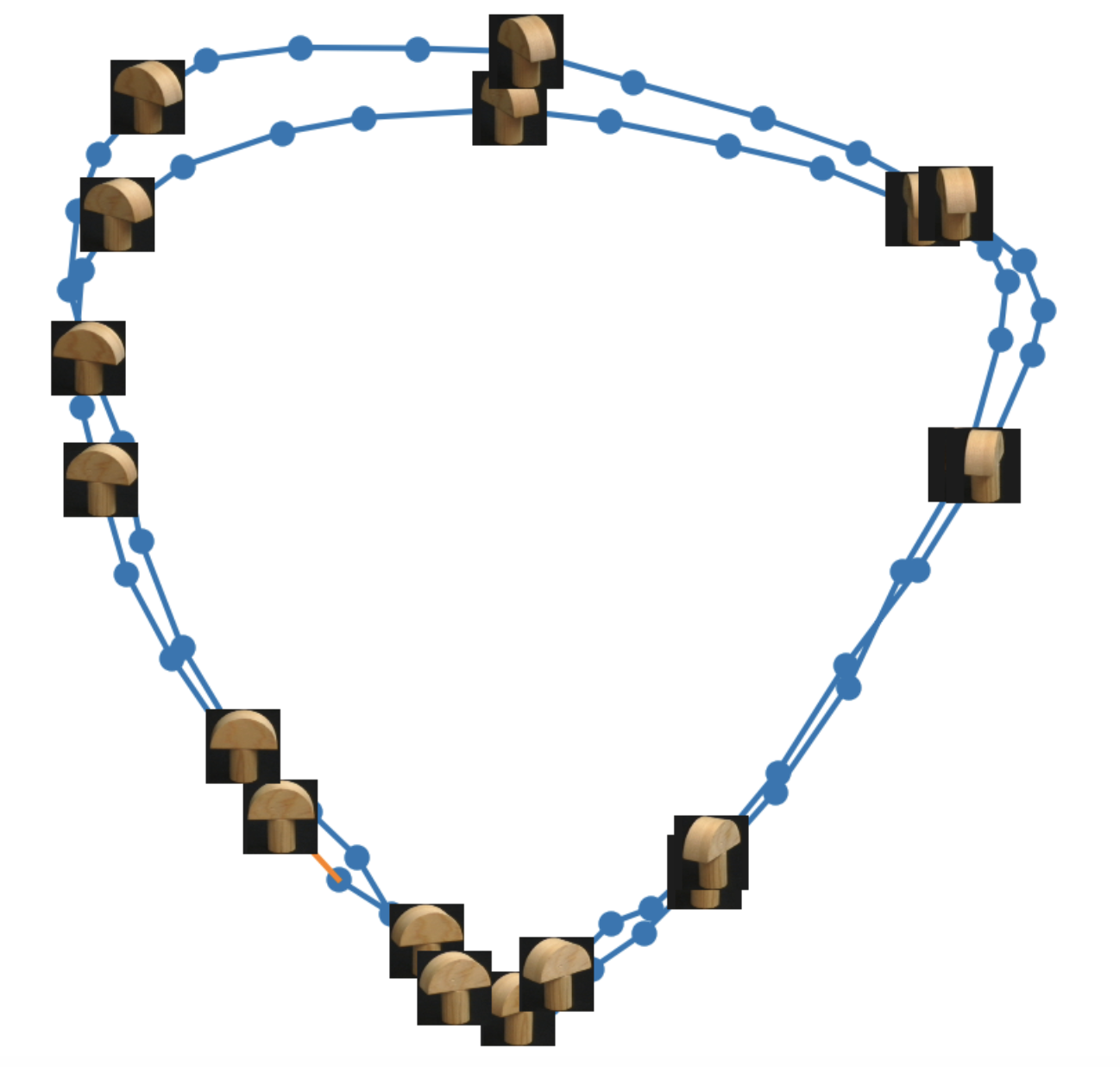} 
\end{tabular}
\qquad & \qquad
\begin{tabular}{c}
\includegraphics[width=0.22\linewidth]{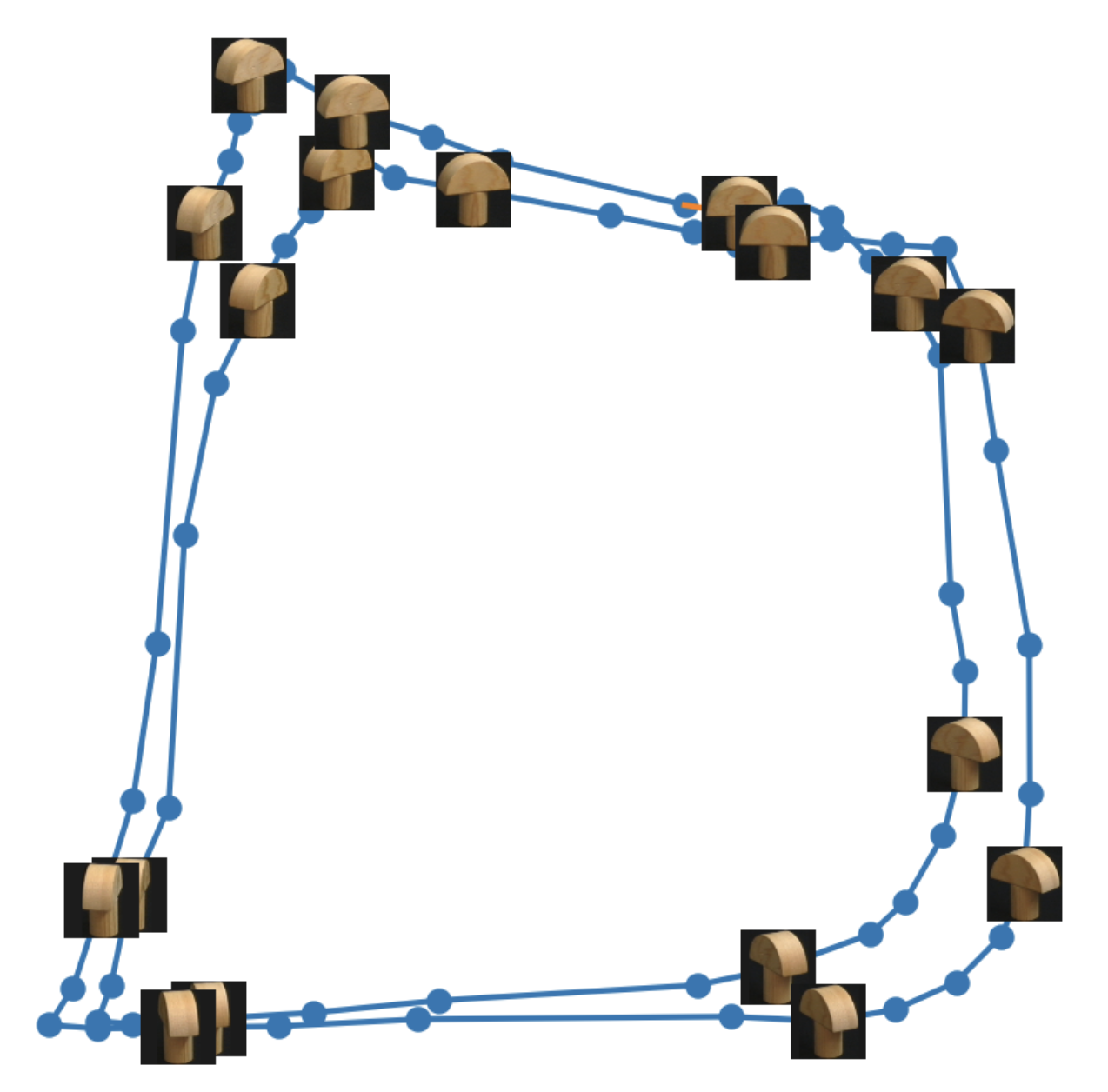}
\end{tabular} \\
(a) PCA \qquad & \qquad (b) POA 
\end{tabular} 
\caption{PCA and POA representations of 72 images taken at 5-degree rotations of a T-shaped object from the COIL-100 database.}
\label{F:2dpoa}
\end{center}
\end{figure}

Figures \ref{F:2dpoa}(a) and \ref{F:2dpoa}(b) show 2-dimensional PCA and POA representations of 72 color images of a T-shaped object from the COIL-100 dataset \cite{nene1996columbia} taken at 5-degree rotations. Each $128 \times 128$ image was converted to grayscale and represented as a vector of pixel values in $\real^{16384}$. To facilitate visualization, consecutive views are connected by an edge and thumbnails of some of the original images are included. As expected, in both cases, the images nearly form loops that are traversed twice because of the symmetry of the object. We emphasize that, even though the original data is Euclidean, the POA representation is calculated on a finite metric space comprising 72 data points. 

\begin{figure}[ht]
\begin{center}
\begin{tabular}{cc}
\begin{tabular}{c}
\includegraphics[width=0.35\linewidth]{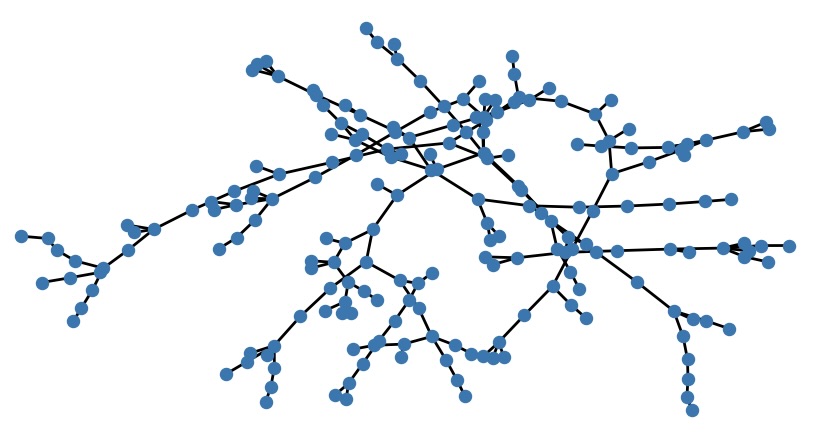}
\end{tabular}
\quad & \quad
\begin{tabular}{c}
\includegraphics[width=0.3\linewidth]{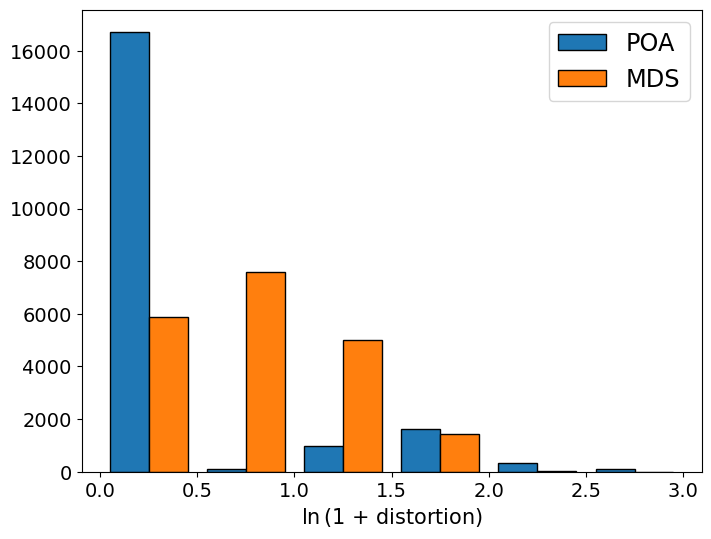}
\end{tabular}
\\
(a) \quad & \quad (b)
\end{tabular}
\caption{(a) A tree with 200 nodes and (b) the histograms of distortions (on a logarithmic scale) for its 3-dimensional POA and MDS embeddings.}
\label{F:comparison}
\end{center}
\end{figure}

POA embeddings into $\real^k$ are not designed to explicitly minimize metric distortions. However, the next example provides some empirical evidence that maximization of variance combined with the 1-Lipschitz condition lead to rather sharp control on metric distortion. We compare POA with the Multidimensional Scaling (MDS) algorithm that seeks to construct an Euclidean embedding that minimizes the average metric distortion (cf.\,\cite{mds1992}). We focus on comparison with MDS because other methods such as the t-Distributed Stochastic Neighbor Embedding (t-SNE) \cite{tsne2008} emphasize other aspects of data visualization, not minimization of distortion. We should also note that comparison of metric distortions of POA and MDS embeddings is not straightforward because the natural metrics on $\real^k$ to consider are not the same, $\|\cdot\|_\infty$ for POA and $\|\cdot\|_2$ for MDS. With these remarks in place, starting with the network depicted in Figure \ref{F:comparison}(a), we assign each edge a unitary length and equip its vertex set $V$ with the shortest path distance $d \colon V \times V \to \real$ and the (normalized) counting measure $\mu$, thus turning $V$ into an $mm$-space. Denote by $\imath_1 \colon V \to \real^3$ and $\imath_2 \colon V \to \real^3$ the 3-dimensional POA and MDS embeddings, respectively, and for each pair of vertices $v,w \in V$, $v \ne w$, calculate the metric distortions
\begin{equation}
\delta_1 (v,w) = \big|d(v,w) - \|\imath_1(v)-\imath_1(w)\|_\infty \big| 
\quad \text{and} \quad
\delta_2 (v,w) = \big|d(v,w) - \|\imath_2(v)-\imath_2(w)\|_2 \big| .
\end{equation}
Figure \ref{F:comparison}(b) shows the histograms of $\ln (1 + \delta_i (v,w))$, $i=1,2$. POA has a much larger number of pairs for which the distortion is small. This can be in part explained by the fact that since MDS seeks to minimize the average distortion, larger distances are likely to influence the outcome more heavily. On the other hand, although in much smaller proportion, POA has more pairs with larger distortions.

%-------------------------

\subsection{Mapping New Data}

In applications, the metric space is typically finite and the probability measure is the (normalized) counting measure, for example, consisting of independent and identically distributed samples from a (possibly unknown) metric-measure space $(X,d,\mu)$. We denote them $X_n = \{x_1, \ldots, x_n\}$ and $\mu_n$, respectively. In such situations, if $\phi \colon X_n \to \real$ is a principal observable, it is important to be able to extend $\phi$ to new ``test'' data points $x \in X \setminus X_n$, as this allows us to estimate their PO coordinates. Thus, the question arises as to whether there is a principled way of extending $\phi$ to new points $x \in X$ maintaining the 1-Lipschitz property. We achieve this through the McShane-Whitney extensions $\phi^+, \phi^- \colon X \to \real$ defined as
\begin{equation} \label{E:wmextension}
\phi^+ (x) := \min_{1 \leq i \leq n} \phi(x_i) + d(x,x_i) \quad \text{and} \quad \phi^- (x) := \max_{1 \leq i \leq n} \phi(x_i) - d(x,x_i).
\end{equation}
$\phi^+$ and $\phi^-$ are the maximal and minimal 1-Lipschitz extension of $\phi$, respectively, meaning that any other 1-Lipschitz extension $\psi$ satisfies $\phi^- (x) \leq \psi(x) \leq \phi^+ (x)$ \cite{mcshane34,petrakis2018}. As any convex combination of $\phi^+$ and $\phi^-$ is also a 1-Lipschitz extension of $\phi$, we choose the balanced extension $\phi^0 = (\phi^+ + \phi^-)/2$.

%--------------------------

\section{Representation of Signals in the Observable Domain} \label{S:basis}

In this section, we propose to use principal observables as basis functions to obtain a new representation of signals on $mm$-spaces that we refer to as representation in the {\em observable domain}. This is analogous to Fourier decomposition of signals on spaces such as compact Riemannian manifolds or the vertex set of a network. Note however that the principle followed in the construction of basis functions is maximization of variance, which is distinct from minimization of energy in the Fourier case. 

Let $\phi_i \colon X \to \real$, $1 \leq i \leq k$, be the first $k$ principal observables of $(X,d,\mu)$ and set $\phi_0 \equiv 1$. Here we allow $k = \infty$ if the $mm$-space $(X,d,\mu)$ admits infinitely many principal observables. We normalize each $\phi_i$, $0 \leq i \leq k$, to get $\mu$-orthonormal basis functions 
\begin{equation}
u_i := \phi_i/\|\phi_i\|_{2,\mu} \,,
\end{equation}
where $\| \cdot \|_{2,\mu}$ denotes the $\mathbb{L}_2$-norm with respect to $\mu$. Using the basis functions $(u_i)_{i=0}^k$, a signal $f \colon X \to \real$ can be represented by the sequence $a = (a_i)_{i=0}^k$ of coefficients given by
\begin{equation}
a_i := \langle f,u_i\rangle_2 = \int_X f(x) u_i (x) \, d\mu(x) \,,
\end{equation}
$0 \leq i \leq k$. The sequence $a=(a_i)_{i=0}^k$ gives a $(k+1)$-dimensional representation of $f$ in the observable domain. As usual, the sequence $a$ is the same for signals that coincide almost everywhere. Note that $f$ is centered if and only if $a_0 =0$. Moreover, we get a principal observable approximation
\begin{equation}
f \sim f_k = \sum_{i=0}^k a_i u_i .
\end{equation}

%-------------------
\begin{figure}[ht]
\begin{center}
\begin{tabular}{cc}
\begin{tabular}{c}
\includegraphics[width=0.7\linewidth]{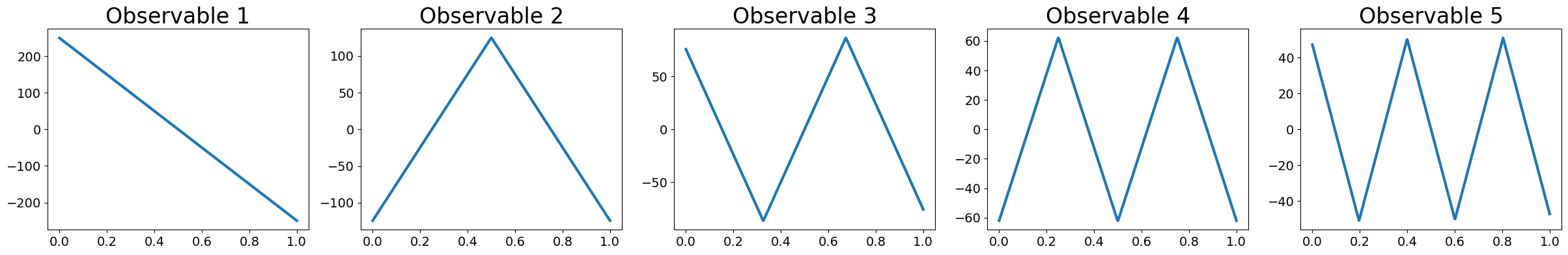} \\
\includegraphics[width=0.7\linewidth]{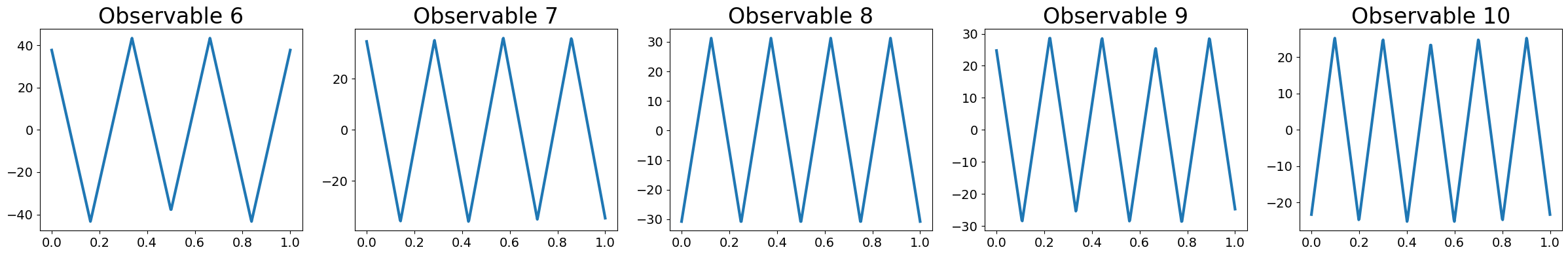}
\end{tabular}
&
\begin{tabular}{c}
\includegraphics[width=0.18\linewidth]{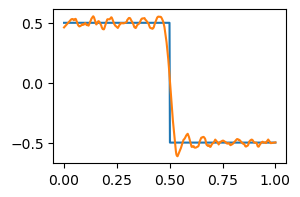} \\
\includegraphics[width=0.18\linewidth]{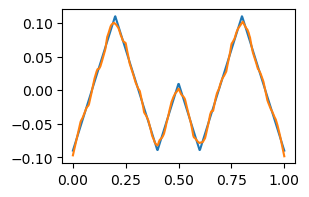}
\end{tabular}
\\
(a) & (b)
\medskip \\
\begin{tabular}{c}
\includegraphics[width=0.7\linewidth]{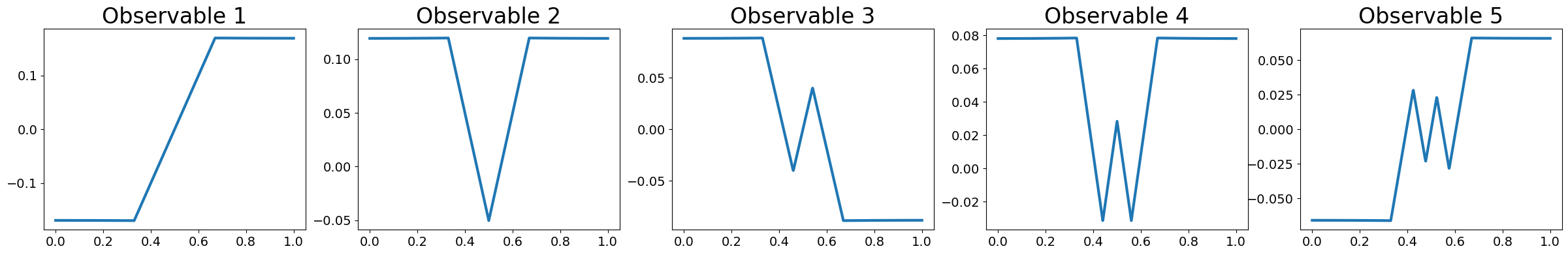} \\
\includegraphics[width=0.7\linewidth]{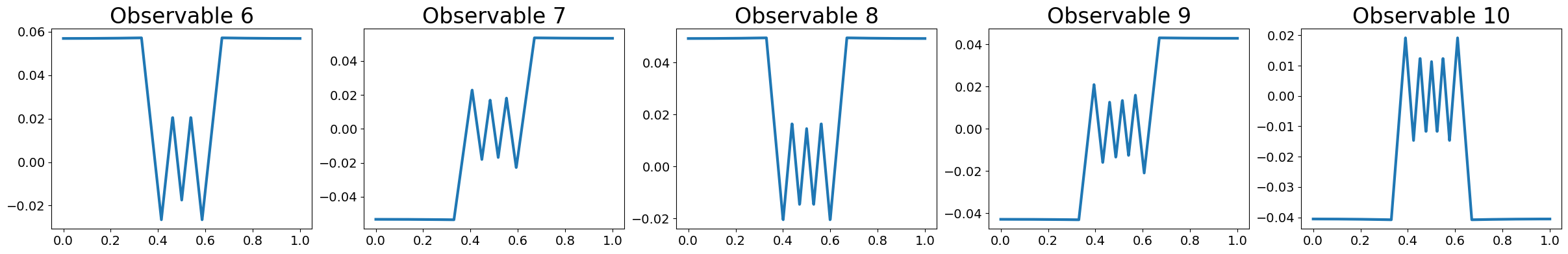}
\end{tabular}
&
\begin{tabular}{c}
\includegraphics[width=0.18\linewidth]{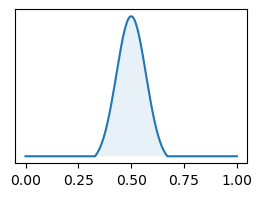} 
\end{tabular}
\\
(c) & (d)
\end{tabular}
\caption{(a) First 10 (normalized) principal observables for a line graph with 501 nodes; (b) signal approximation using the first 25 principal observables; (c) first 10 (normalized) principal observables for a distribution with more localized support shown in (d). Vertical axis scales differ to facilitate visualization.}
\label{F:basis}
\end{center}
\end{figure}

Figure \ref{F:basis}(a) shows the first ten principal observables for a line graph $L$ with 501 nodes forming a uniform grid on the interval $I=[0,1]$ with consecutive nodes connected by an edge of length $1/500$. The observables were calculated as described in Section \ref{S:poa}. Figure \ref{F:basis}(b) depicts the reconstruction of two signals (restricted to the nodes of the line graph $L$) with 25 basis functions. More precisely, the approximation is given by $f \sim \sum_{i=1}^{25}  a_i  u_i$, where $a_i=\langle f,u_i\rangle_2$. The original signal is in blue and the approximation in orange. Figure \ref{F:basis}(c) shows the first ten principal observables for a more localized bell-shaped distribution depicted in Figure \ref{F:basis}(d). We observe an oscillatory pattern similar to that for the uniform distribution but confined to the support of the distribution.

%---------------

\section{Stability Theorems} \label{S:proof}

We now address the stability of the observable mean and observable covariance for distributions on a fixed compact metric space $(X,d)$. Stability is established with respect to the Wasserstein distance in $(X,d)$. We start with the mean $M_\mu \colon \obs{X} \to \real$ which is simpler to analyze because the domain $\obs{X}$ is independent of the distribution $\mu$. The stability of the covariance $\Sigma_\mu \colon \cobs{\mu} \times \cobs{\mu} \to \real$ is more delicate because the space $\cobs{\mu}$ of centered observables depends on $\mu$.

\subsection{Stability of the Observable Mean} \label{S:stabmean}

To frame the stability of $M_\mu$, we first set up some notation and in the process review the definition of the Wasserstein $p$-distance, $p\geq 1$. Given $\mu, \nu \in \borel(X,d)$, a coupling between $\mu$ and $\nu$ is a probability measure $h$ on the product space $X \times X$ whose marginals are $\mu$ and $\nu$, respectively; that is, ${\pi_1}_\sharp (h)=\mu$ and ${\pi_2}_\sharp (h)=\nu$, where $\pi_1,\pi_2 \colon X \times X \to X$ are the projections onto the first and second coordinates. The collection of all such couplings is denoted $\Gamma(\mu,\nu)$. If $h \in \Gamma(\mu,\nu)$, let
\begin{equation} \label{E:pcost}
E_p (h) = \Big(\int_{X \times X} d^p(x,y) dh(x,y)\Big)^{1/p}.
\end{equation}
The {\em Wasserstein $p$-distance} between $\mu$ and $\nu$ is given by $w_p(\mu,\nu) = \inf_{h \in \Gamma(\mu,\nu)} E_p (h)$. We prove stability of $M_\mu$ with respect to $w_1$, as this implies stability with respect to $w_p$ because $w_1 (\mu,\nu) \leq w_p (\mu,\nu)$, for any $p >1$.

\begin{theorem} \label{T:stabmean}
If $\mu, \nu \in \borel (X,d)$ are Borel probability measures on a compact metric space $(X,d)$, then
\[
\|M_\mu - M_\nu\|_\infty = \sup_{f\in \obs{X}} |M_\mu f - M_\nu f| \leq w_1 (\mu,\nu).
\]
\end{theorem}

\begin{proof}
 Given $\epsilon > w_1 (\mu,\nu)$, let $h \in \Gamma(\mu,\nu)$ be such that $E_p(h) < \epsilon$. For any $f \in \obs{X}$, since the marginals of $h$ are $\mu$ and $\nu$, we may write
\begin{equation} 
M_\mu f = \int_X f(x) \, d\mu (x) = \int _{X \times X} f(x) \,dh (x,y)
\end{equation}
and
\begin{equation} 
M_\nu f = \int_X f(y) \, d\nu (y) = \int _{X \times X} f(y) \,dh (x,y) \,.
\end{equation}
Thus, 
\begin{equation} \label{E:mean}
|M_\mu f - M_\nu f| \leq \int_{X \times X} |f(x)-f(y)|\, dh(x,y) \leq \int_{X \times X} d(x,y) \, dh(x,y) < \epsilon \,.
\end{equation}
The second inequality in \eqref{E:mean} uses the fact that $f$ is 1-Lipschitz. Since \eqref{E:mean} holds for any $\epsilon > w_1 (\mu,\nu)$ and $f \in \obs{X}$ is arbitrary, the claim follows.
\end{proof}

%----------------------------------

\subsection{Stability of the Observable Covariance} 

The domains of the covariance operators $\Sigma_\mu \colon \cobs{\mu} \times \cobs{\mu} \to \real$ and $\Sigma_\nu \colon \cobs{\nu} \times \cobs{\nu} \to \real$ for $\mu, \nu \in \borel (X,d)$  in general differ. Nonetheless, as $\cobs{\mu} \times \cobs{\mu}$ and $\cobs{\nu} \times \cobs{\nu}$ are both subspaces of $\obs{X} \times \obs{X}$, to compare $\Sigma_\mu$ and $\Sigma_\nu$, we use a (functional) Hausdorff distance $d_H (\Sigma_\mu, \Sigma_\nu)$ that is described next in broader generality.

\begin{definition} \label{D:hausdorff}
Let $(Z,d_Z)$ be a metric space, and $\phi \colon A \to \real$ and $\psi \colon B \to \real$ be continuous functions defined on compact subspaces $A, B \subseteq Z$. 
\begin{enumerate}
\item A real number $\epsilon >0$ is {\em admissible} for $(\phi,\psi)$ if the following holds:
\begin{enumerate}
\item for any $a \in A$, there is $b \in B$ such that $d_Z(a,b) < \epsilon$ and $|\phi(a)-\psi(b)| < \epsilon$;
\item for any $b \in B$, there is $a \in A$ such that $d_Z(a,b) < \epsilon$ and $|\phi(a)-\psi(b)| < \epsilon$;
\end{enumerate}
\item The {\em Hausdorff distance} between $\phi$ and $\psi$ is defined as
\[
d_H (\phi,\psi) = \inf \{\epsilon>0 \colon \text{$\epsilon$ admissible for $(\phi,\psi)$}\}.
\]
\end{enumerate}
\end{definition}

This distance may be interpreted as a standard Hausdorff distance between compact subsets of a metric space \cite{hsdrff1920} via graphs. The graph of $\phi$ is given by
\begin{equation}
\Gamma_\phi = \{(a,\phi(a)) \in Z \times \real \colon a \in A\} \subseteq Z \times \real
\end{equation}
and $\Gamma_\psi \subseteq Z \times \real$ is defined similarly. Equip the product space $Z \times \real$ with the metric
\begin{equation}
d((z,t), (z',t')) = \max \{d_Z(z,z'), |t-t'|\}.
\end{equation}
Then, $d_H(\phi,\psi)$ is the Hausdorff distance in $Z \times \real$ between the compact subspaces $\Gamma_\phi$ and $\Gamma_\psi$. That is, $d_H (\phi,\psi) = d_H^{Z \times \real} (\Gamma_\phi, \Gamma_\psi)$. 

The next lemma is in preparation for the proof of a stability theorem for the observable covariance. Below, we use the abbreviation $f_\mu \coloneq f - M_\mu f$.

\begin{lemma} \label{L:estimate}
If $\mu, \nu \in \borel (X,d)$ and $f,g \in \obs{X}$, then $\big|\Sigma_\mu (f_\mu , g_\mu ) - \Sigma_\nu (f_\nu , g_\nu )\big| \leq 4 D_X \cdot w_1 (\mu,\nu)$, where $D_X$ is the diameter of $X$.
\end{lemma}

\begin{proof}
Let $h \in \Gamma(\mu, \nu)$ and $E_1(h) = \int_X d(x,y) dh (x,y)$. Letting $I (h) = \big|\Sigma_\mu (f_\mu,g_\mu) - \Sigma_\nu (f_\nu ,g_\nu ) \big|$, we have
\begin{equation} \label{E:covestimate1}
\begin{split}
I(h) &= \big|\int_X f_\mu (x) g_\mu (x) d\mu(x) - \int_X f_\nu(y) g_\nu (y) d\nu(y) \big| \\
&= \big|\int_{X\times X} f_\mu (x) g_\mu (x) dh(x,y) - \int_{X \times X} f_\nu(y) g_\nu (y) dh (x,y) \big| \\
&\leq \int_{X\times X} |f_\mu (x)| \cdot  |g_\mu (x)-g_\nu (x)| dh(x,y) 
+ \int_{X\times X} |f_\mu (x)| \cdot |g_\nu (x)-g_\nu (y)| dh(x,y) \\
&+ \int_{X\times X} |f_\mu (x)-f_\mu(y)| \cdot |g_\nu (y)| dh(x,y) 
+ \int_{X\times X} |f_\mu (y)-f_\nu(y)| \cdot |g_\nu (y)| dh(x,y).
\end{split}
\end{equation}
From \eqref{E:covestimate1}, using Lemma \ref{L:bound} and the fact that observables are 1-Lipschitz, we obtain
\begin{equation} \label{E:covestimate2}
\begin{split}
I(h) \leq 2 D_X \int_{X \times X} d(x,y) dh(x,y)
&+ D_X \int_{X\times X} |g_\mu (x)-g_\nu (x)| dh(x,y) \\
&+ D_X \int_{X\times X} |f_\mu (y)-f_\nu(y)| dh(x,y).
\end{split}
\end{equation}
Since $f_\mu(y) - f_\nu(y) = M_\mu f-M_\nu f$ and $g_\mu (x) - g_\nu (x) = M_\mu g - M_\nu g$, inequality \eqref{E:covestimate2} and Theorem \ref{T:stabmean} yield
\begin{equation}
I(h) \leq 2 D_X \int_{X \times X} d(x,y) dh(x,y) + 2 D_X w_1 (\mu,\nu).
\end{equation}
As the coupling $h$ is arbitrary, we can conclude that $\big|\Sigma_\mu (f_\mu , g_\mu ) - \Sigma_\nu (f_\nu , g_\nu )\big| \leq 4 D_X \cdot w_1 (\mu,\nu)$, as claimed.
\end{proof}

\begin{theorem} \label{T:stabcov}
Let $(X,d)$ be a compact metric space. If $\mu, \nu \in \borel (X,d)$, then $d_H(\Sigma_\mu, \Sigma_\nu) \leq C_X \cdot w_1 (\mu,\nu)$, where $C_X = \max\{1, 4D_X\}$.
\end{theorem}

\begin{proof}
We first show that any $\epsilon > C_X \cdot w_1 (\mu,\nu)$ is admissible for $\Sigma_\mu$ and $\Sigma_\nu$. To verify condition 1(a) in Definition \ref{D:hausdorff}, given $(f,g) \in \cobs{\mu} \times \cobs{\mu}$, let $(f',g') = (f_\nu, g_\nu) \in \cobs{\nu} \times \cobs{\nu}$. Since $M_\mu f =0$, we have that $f_\mu-f_\nu = (f-M_\mu f) -(f-M_\nu f) = M_\nu f- M_\mu f$ and similarly for $g$. Thus,
\begin{equation}
\begin{split}
\|(f,g)-(f',g')\|_\infty = \max \big\{ \|f_\mu-f_\nu\|_\infty, \|g_\mu -g_\nu\|_\infty \big\}
= \max \big\{ |M_\mu f-M_\nu f|, |M_\mu g-M_\nu g| \big\} .
\end{split}
\end{equation}
Thus, by Theorem \ref{T:stabmean}, $\|(f,g)-(f',g')\|_\infty \leq w_1(\mu,\nu) < \epsilon$, verifying the first part of condition 1(a). The second part of 1(a) also holds because, by Lemma \ref{L:estimate},
\begin{equation}
\big|\Sigma_\mu (f,g) - \Sigma_\nu (f',g')\big| =  \big|\Sigma_\mu (f_\mu,g_\mu) - \Sigma_\nu (f_\nu,g_\nu) \big| 
\leq 4 D_X \cdot w_1 (\mu,\nu) < \epsilon.
\end{equation}
Similarly, we verify condition 1(b) in Definition \ref{D:hausdorff}. Since $\epsilon > C_X \cdot w_1(\mu,\nu)$ is arbitrary, it follows that $d_H(\Sigma_\mu, \Sigma_\nu) \leq C_X \cdot w_1 (\mu,\nu)$.
\end{proof}

%----------------------------------

\subsection{Consistency} 

In practice, the distribution $\mu$ is typically unknown and we only have access to random draws from $\mu$. As such, we examine how well the observable mean and covariance can be estimated from empirical data. More formally, let $x_i \in X$, $1 \leq i < \infty$, be a sequence of random draws from $\mu$ and, for each $n \geq 1$, let $\mu_n = \sum_{i=1}^n \delta_{x_i}/n$ be the associated empirical measure. 

\begin{corollary}
If $x_i \in X$, $i \geq 1$, is a sequence of independent random draws from $\mu$, then:
\begin{enumerate}[\rm (i)]
\item $\lim_{n \to \infty} |M_\mu (f) - M_{\mu_n}(f)|= 0$ uniformly on $f \in \obs{X}$, ${\otimes_\infty}\mu$-almost surely;
\item $\lim_{n \to \infty} d_H(\Sigma_\mu, \Sigma_{\mu_n}) = 0,$  ${\otimes_\infty}\mu$-almost surely.
\end{enumerate}
\end{corollary}

\begin{proof}
This is a direct consequence of the stability results obtained in Theorem \ref{T:stabmean} and Theorem \ref{T:stabcov}, and the facts that (a) the Wasserstein distance $w_1$ metrizes weak convergence of probability measures  and (b) $\mu_n \to \mu$ weakly and almost surely \cite{dudley2002real}.
\end{proof}

\noindent
{\em Remark.} Rates of convergence for the empirical observable mean and the empirical observable covariance can be derived from estimates for Wasserstein convergence of empirical measures due to Weed and Bach \cite{weed2019}.

%----------------

\section{Stability for Heterogeneous Data} \label{S:heterogeneous}

In problems such as analysis and integration of heterogeneous data (e.g., data of different types or from distinct sources), it is important to have estimates for how much information is gained from data of different origin. Theoretically, such problems can be framed as defining distances or measures of divergence between probability measures on different domains. The $L_p$ transportation distance of Sturm, $p \geq 1$, is one such measure of discrepancy or information gained \cite{sturm2006}. We denote this distance by $\dks$ and, as in \cite{mioneedham25}, call it the Kantorovich-Sturm distance. The reason for this terminology is that some authors have referred to $\dks$ as the Gromov-Wasserstein distance but this is not the same as the Gromov-Wasserstein distance defined by M\'{e}moli \cite{memoli2011}, which is a variant of $\dks$. The main goal of this section is to prove that the observable mean and observable covariance are stable with respect to $\dks$, thus satisfying a form of stability even in the realm of heterogeneous data.

Let $\mm{M} = (M,d,\mu)$ and $\mm{M}' = (M',d',\mu')$ be $mm$-spaces, where $M$ and $M'$ are compact, and denote by $K \subseteq M$ and $K' \subseteq M'$ the supports of $\mu$ and $\mu'$, respectively. A {\em metric coupling} between $\mm{M}$ and $\mm{M}'$ is a metric $\delta \colon (M \sqcup M') \times (M \sqcup M') \to \real$ on the disjoint union $M \sqcup M'$ such that $\delta|_{K\times K} = d|_{K\times K}$ and  $\delta|_{K' \times K'} = d'|_{K'\times K'}$. The set of all such metric couplings is denoted $C((d,\mu),(d',\mu'))$. A {\em probabilistic coupling} between $\mu$ and $\mu'$ is a Borel probability measure $h$ on $M \times M'$ that marginalizes to $\mu$ and $\mu'$, respectively. As in Section \ref{S:stabmean}, the set of all probabilistic couplings is denoted $\Gamma(\mu,\mu')$. Note that the support of any coupling $h$ satisfies $\text{supp}\,[h] \subseteq K \times K'$. Therefore,
\begin{equation}
\int_{X\times X'} \phi(x,y) dh(x,y) = \int_{K \times K'} \phi(x,y) dh(x,y),
\end{equation}
for any integrable function $\phi$, a fact that we use repeatedly below.

\begin{definition}[\cite{sturm2006}] \label{D:dks}
The {\em Kantorovich-Sturm $p$-distance}, $p \geq 1$, is defined as
\[
\dks (\mm{M},\mm{M}') \coloneq \inf_{\substack{h \in \Gamma(\mu,\mu')\\ \delta \in C((d,\mu),(d',\mu'))}} \Big( \int_{M \times M'} \delta^p (z,z') \,dh (z,z') \Big)^{1/p}.
\]
\end{definition}

To compare observable means and observable covariance operators, we employ a functional version of the Gromov-Hausdorff distance $\dgh$ between bounded (i.e., finite diameter) pseudo metric spaces, analogous to that of \cite{curry2024,anbouhi2024}. We need $\dgh$ in the generality of bounded pseudo metric spaces to address stability. Let $(Z,d)$ and $(Z',d')$ be bounded pseudo metric spaces and $f \colon Z \to \real$ and $f' \colon Z' \to \real$ continuous functions. We refer to the triples $\mm{Z} = (Z,d,f)$ and $\mm{Z}' = (Z',d',f')$ as functional (pseudo) metric spaces. A {\em correspondence} between $Z$ and $Z'$ is a relation $R \subseteq Z \times Z'$ such that $\pi_Z (R) = Z$ and $\pi_{Z'} (R) = Z'$, where $\pi_Z$ and $\pi_{Z'}$ denote projections.

\begin{definition} \label{D:dgh}
The {\em structural distortion} $dis(R)$ and the {\em functional distortion} $fdis(R)$ of the correspondence $R$ are given by
\[
dis(R) \coloneq \sup_{\substack{(z_1,z'_1) \in R\\ (z_2,z'_2) \in R}} |d(z_1, z_2) - d'(z'_1,z'_2)| 
\quad \text{and} \quad
fdis(R) \coloneq \sup_{(z,z') \in R} |f(z)-f'(z')|,
\]
respectively. The {\em Gromov-Hausdorff distance} between $f$ and $f'$ is defined as
\[
\dgh (f,f') \coloneq \frac{1}{2} \inf_R \max \big\{dis (R), 2 fdis(R) \big\},
\]
where the infimum is taken over all correspondences $R \subseteq Z \times Z'$. 
\end{definition}

%------------------

\subsection{Heterogeneous Stability of the Mean}

Let $(X,d,\mu)$ be an $mm$-space. To establish the stability of the observable mean $M_\mu \colon \obs{X} \to \real$ with respect to the Gromov-Wasserstein $p$-distance, $p \geq 1$, we first define a family of pseudo metric on $\obs{X}$ induced by the $L_p$ norms with respect to $\mu$. Define the pseudo metric $\dlp{\mu} \colon \obs{X} \times \obs{X} \to \real$ by
\begin{equation} \label{E:dlp}
\dlp{\mu} (f,g) \coloneq \|f-g\|_{p,\mu} = \Big(\int_X |f(x)-g(x)|^p d\mu (x)\Big)^{1/p}.
\end{equation}
We adopt the notation $\obsp{X}$ for the pseudo metric space $(\obs{X}, \dlp{\mu})$ and $\meanp{\mu} \colon \obsp{X} \to \real$ for the observable mean. The boundedness of $\obsp{X}$ and the continuity of $\meanp{\mu}$ are simple to verify.

\begin{theorem} \label{T:hmean}
If $(X,d,\mu)$ and $(X',d',\mu')$ are $mm$-spaces, $X$ and $X'$ are compact, and $p\geq 1$, then
\[
\dgh(\meanp{\mu}, \meanp{\mu'}) \leq \dks(\mu,\mu').
\]
\end{theorem}

The next definition and lemma are in preparation for the proof of Theorem \ref{T:hmean}. Given $f \in \obs{\mu}$, we denote by $f_K \colon K \to \real$ the restriction of $f$ to the support $K$ of $\mu$; that is, $f_K = f|_K$.

\begin{definition} \label{D:mcoupling}
Given $\delta \in C((d,\mu),(d',\mu'))$, define a relation $R_\delta \subseteq \obsp{X} \times \obsp{X'}$ as the collection of pairs $(f,f') \in \obsp{X} \times \obsp{X'}$ such that $f_K \sqcup f'_{K'} \colon (K \sqcup K', \delta) \to \real$ is 1-Lipschitz. In words, $(f,f') \in R_\delta$ if $f_K \sqcup f'_{K'}$ is an observable for the disjoint union $K \sqcup K'$ equipped with the metric $\delta$.
\end{definition}

\begin{lemma} \label{L:correspondence}
If $\delta \in C((d,\mu),(d',\mu'))$ is a metric coupling, then $R_\delta$ is a correspondence.
\end{lemma}

\begin{proof}
We need to show that the projections of $R_\delta$ to $\obsp{X}$ and $\obsp{X'}$ are surjective. Let $f \in \obsp{X}$. Define $f'_{K'} \colon K' \to \real$ by
\begin{equation}
f'_{K'}(x') \coloneq \inf_{x \in K} f(x) + \delta(x,x'),
\end{equation}
for any $x' \in K'$. The map $f_K \sqcup f'_{K'} \colon K \sqcup K' \to \real$ is the (upper) McShane-Whitney extension of $f_K$ to $K \sqcup K'$, which is well-defined (because $K$ is compact) and 1-Lipschitz \cite{mcshane34,petrakis2018}. Use another McShane-Whitney construction to extend $f'_{K'} \colon K' \to \real$ to a 1-Lipschitz function $f' \colon X' \to \real$. Then, $(f,f') \in R_\delta$, showing that the projection to $\obsp{X}$ is surjective. The same argument applies to the projection to $\obsp{X'}$.
\end{proof}

\begin{proof}[\bf Proof of Theorem \ref{T:hmean}]
To obtain upper bounds for the Gromov-Hausdorff distance between the observable means $\meanp{\mu} \colon \obsp{\mu} \to \real$ and $\meanp{\mu'} \colon \obsp{\mu'} \to \real$, let $\delta \in C((d,\mu), (d',\mu'))$, $h \in \Gamma(\mu,\mu')$, and $R_\delta \subseteq \obsp{\mu}\times \obsp{\mu'}$ the correspondence described in Definition \ref{D:mcoupling}. To estimate the structural distortion of $R_\delta$, let $(f,f'), (g,g') \in R_\delta$. We have
\begin{equation} 
\dlp{\mu} (f,g)= \Big(\int_X |f(x)-g(x)| d\mu (x) \Big)^{1/p}
= \Big(\int_{X \times X'} |f(x)-g(x)| dh (x,y) \Big)^{1/p}
= \|f-g\|_{p,h}
\end{equation}
and similarly $\dlp{\mu'}(f',g') = \|f'-g'\|_{p,h}$. Therefore,
\begin{equation} \label{E:rpdis}
\begin{split}
|\dlp{\mu} (f,g) &- \dlp{\mu'}(f',g')| = \big|\|f-g\|_{p,h} - \|f'-g'\|_{p,h}\big| \leq \|f-f'\|_{p,h} + \|g-g'\|_{p,h} \\
&= \Big(\int_{X \times X'} |f(x)-f'(y)| dh (x,y) \Big)^{1/p} + 
\Big(\int_{X \times X'} |g(x)-g'(y)| dh (x,y) \Big)^{1/p} \\
&= \Big(\int_{K \times K'} |f(x)-f'(y)| dh (x,y) \Big)^{1/p} + 
\Big(\int_{K \times K'} |g(x)-g'(y)| dh (x,y) \Big)^{1/p} \\
&\leq 2 \Big(\int_{K \times K'} \delta^p(x,y) dh (x,y) \Big)^{1/p}
= 2 \Big(\int_{X \times X'} \delta^p(x,y) dh (x,y) \Big)^{1/p},
\end{split}
\end{equation}
where we have used the fact that $f_k\sqcup f'_{K'}$ and $g_K \sqcup g'_{K'}$ are 1-Lipschitz with respect to $\delta$. Since $(f,f'), (g,g') \in R_\delta$ are arbitrary, we obtain
\begin{equation} \label{E:sdist}
dis(R_\delta) \leq 2 \Big(\int_{X \times X'} \delta^p(x,y) dh (x,y) \Big)^{1/p}.
\end{equation}
For the functional distortion, if $(f,f') \in R_\delta$, we have
\begin{equation} \label{E:meandif}
\begin{split}
|\meanp{\mu} (f) - \meanp{\mu'} (f')| &= \Big|\int_X f(x) d\mu(x) - \int_{X'} f'(y) d\mu' (y)\Big| \\
&= \Big|\int_{X\times X'} f(x) dh(x,y) - \int_{X \times X'} f'(y) dh (x,y)\Big|\\
&\leq \int_{X\times X'} |f(x)-f'(y)| dh(x,y)
= \int_{K\times K'} |f(x)-f'(y)| dh(x,y)\\
&\leq \int_{K\times K'} \delta(x,y) dh (x,y)
= \int_{X\times X'} \delta (x,y) dh (x,y).
\end{split}
\end{equation}
Therefore, for any $p\geq 1$,
\begin{equation} \label{E:fdist}
fdis(R_\delta) \leq \int_{X\times X'} \delta (x,y) dh (x,y)
\leq \Big(\int_{X\times X'} \delta^p (x,y) dh (x,y)\Big)^{1/p}.
\end{equation}
From \eqref{E:sdist} and \eqref{E:fdist}, we obtain
\begin{equation}
\dgh(\meanp{\mu}, \meanp{\mu'}) \leq \frac{1}{2} \max\{ dis (R_\delta), 2 fdis(R_\delta)\} \leq 
\Big(\int_{X\times X'} \delta^p (x,y) dh (x,y)\Big)^{1/p}.
\end{equation}
This in turn implies that $\dgh(\meanp{\mu}, \meanp{\mu'}) \leq \dks(\mu,\mu')$ because the couplings $\delta$ and $h$ are arbitrary.
\end{proof}

%------------------

\subsection{Heterogeneous Stability of the Covariance}

To address the stability of the observable covariance $\Sigma_\mu \colon \cobs{\mu}\times \cobs{\mu} \to \real$, we introduce a family of pseudo metrics on $\cobs{\mu} \times \cobs{\mu}$, indexed by $p \geq 1$, that turn $\cobs{\mu} \times \cobs{\mu}$ into bounded pseudo metric spaces. For $(f,g), (f',g') \in \cobs{\mu} \times \cobs{\mu}$, define the product (pseudo) metric
\begin{equation} \label{E:dlpc}
\dlpc{\mu} ((f,f'), (g,g')) \coloneq
\max \{\dlp{\mu} (f,g), \dlp{\mu} (f',g')\},
\end{equation}
with $\dlp{\mu}$ as in \eqref{E:dlp}. We use the abbreviation $\pcobsp{\mu}= \cobs{\mu} \times \cobs{\mu}$ and the notation $\covp{\mu} \colon \pcobsp{\mu} \to \real$ to indicate that the domain of $\Sigma_\mu$ is equipped with the distance $\dlpc{\mu}$. 

\begin{theorem} \label{T:hcov}
If $(X,d,\mu)$ and $(X',d',\mu')$ are $mm$-spaces with $X$ and $X'$ compact, then
\[
\dgh(\covp{\mu}, \covp{\mu'}) \leq C(X,X') \dks(\mu,\mu'),
\]
for any $p \geq 1$. Here, $C(X,X')= \max\{2, 2(D_X+D_{X'})\}$, where $D_X$ and $D_{X'}$ are the diameters of $(X,d)$ and $(X',d')$, respectively.
\end{theorem}

Similar to the argument in the proof of the heterogeneous stability of the observable mean, we construct correspondences between the domains $\pcobsp{\mu} = \cobs{\mu} \times \cobs{\mu}$ and $\pcobsp{\mu'} = \cobs{\mu'} \times \cobs{\mu'}$ of the covariance operators starting with a correspondence between the factors $\cobs{\mu}$ and $\cobs{\mu'}$. 

\begin{definition} \label{D:paircorr}
Let $\delta \in C((d,\mu),(d',\mu'))$ be a metric coupling. Define $S_\delta \subseteq \cobs{\mu} \times \cobs{\mu'}$ as the collection of all pairs $(F,F') \in \cobs{\mu} \times \cobs{\mu'}$ such that there exist $(f,f') \in \obs{X} \times \obs{X'}$ such that $f_k \sqcup f'_{K'} \colon (K \sqcup K', \delta) \to \real$ is 1-Lipschitz, $F = f-\mean{\mu}f$ and $F'=f'-\mean{\mu'}f'$.
\end{definition}

The relation $S_\delta$ simply takes any pair of observables $(f,g)$ that defines an observable on the disjoint union $K \sqcup K'$ and centers its restrictions to $X$ and $X'$. That $S_\delta$ is a correspondence can be verified as in the proof of Lemma \ref{L:correspondence} using McShane-Whitney extensions.

\begin{lemma} \label{L:sdis}
If $\delta \in C((d,\mu),(d',\mu'))$ is a metric coupling, then the structural distortion of $S_\delta$ satisfies 
\[
dis(S_\delta) \leq 4 \Big( \int_{X \times X'} \delta^p (x,y) dh (x,y) \Big)^{1/p},
\]
for any probabilistic coupling $h \in \Gamma(\mu,\mu')$.
\end{lemma}

\begin{proof}
We adopt the abbreviations $M = \mean{\mu}$ and $M' = \mean{\mu'}$ for the observable mean operators. Let $(F,F'), (G,G') \in S_\delta$. By the definition of $S_\delta$, there are observables $f,g \in \obs{X}$ and $f',g' \in \obs{X'}$ such that $f_K \sqcup f'_{K'}\colon (K \sqcup K', \delta) \to \real$ and $g_K \sqcup g'_{K'}\colon (K \sqcup K', \delta) \to \real$ are 1-Lipschitz, $(F,F') = (f-Mf, f'-M'f')$ and $(G,G')= (g-Mg, g'-M'g')$. Then,
\begin{equation} \label{E:dist1}
\begin{split}
&\quad \ \big|\dlp{\mu} (F,G) - \dlp{\mu'} (F',G')\big| 
= \big|\|F-G\|_{p,\mu} - \|F'-G'\|_{p,\mu'}\big|\\
&= \big|\| (f-Mf) - (g-Mg)\|_{p,\mu} - \| (f'-M'f') - (g'-M'g')\|_{p,\mu'}\big| \\
&= \big|\| (f-Mf) - (g-Mg)\|_{p,h} - \| (f'-M'f') - (g'-M'g')\|_{p,h}\big| \\
&\leq \|f-f'\|_{p,h} + |Mf-M'f'| +  \|g-g'\|_{p,h} + |Mg - M'g'|.
\end{split}
\end{equation}
Arguing as in \eqref{E:rpdis} to estimate the terms $\|f-f'\|_{p,h}$ and $\|g-g'\|_{p,h}$, and using the upper bounds for $|Mf-M'f'|$ and $|Mg-M'g'|$ obtained in \eqref{E:meandif}, inequality \eqref{E:dist1} implies that
\begin{equation}
\big|\dlp{\mu} (F,G) - \dlp{\mu'} (F',G')\big| \leq
4 \Big( \int_{X \times X'} \delta^p (x,y) dh (x,y) \Big)^{1/p}.
\end{equation}
Since $(F,G)$ and $(F',G')$ are arbitrary, the lemma follows.
\end{proof}

\begin{lemma} \label{L:fdis}
If $\delta \in C((d,\mu),(d',\mu'))$ is a metric coupling, then
\[
fdis(\hat{S}_\delta) = \sup_{\substack{(F,F') \in S_\delta \\ (G,G') \in S_\delta}}
|\Sigma_\mu (F,G) - \Sigma_{\mu'} (F',G')| \leq 2(D_X + D_{X'})
\int_{X \times X'} \delta (x,y) dh (x,y),
\]
for any probabilistic coupling $h \in \Gamma(\mu,\mu')$.
\end{lemma}

\begin{proof}
See Appendix \ref{A:fdis}.
\end{proof}

As the domains of the covariance operators $\Sigma_\mu$ and $\Sigma_{\mu'}$ are the product spaces $\pcobs{\mu} = \cobs{\mu} \times \cobs{\mu}$ and $\pcobs{\mu'} = \cobs{\mu'} \times \cobs{\mu'}$, in addressing stability, we consider the product correspondence $\hat{S}_\delta \subseteq \pcobs{\mu} \times \pcobs{\mu'}$ given by $\hat{S}_\delta \coloneq S_\delta \times S_\delta$. This means that
\begin{equation} \label{E:prodcor}
((F,G),(F',G')) \in \hat{R}_\delta \iff (F,F') \in R_\delta 
\ \text{and} \ (G,G') \in R_\delta.
\end{equation}

\begin{lemma} \label{L:proddis}
Let $\delta \in C((d,\mu),(d',\mu'))$. The structural distortion of the product correspondence $\hat{S}_\delta$ satisfies $dis(\hat{S}_\delta) = dis (S_\delta)$.
\end{lemma}

\begin{proof}
If $(F,F'), (G,G') \in S_\delta$, then $((F,F'),(F,F')), ((G,G'),(G,G')) \in \hat{S}_\delta$ and we have
\begin{equation}
\begin{split}
\big|\dlp{\mu} (F,G) - \dlp{\mu'} (F',G')\big| 
&= \big|\|F-G\|_{p,\mu} - \|F'-G'\|_{p,\mu'}\big| \\
&= \big|\dlpc{\mu} ((F,F),(G,G)) - \dlpc{\mu'} ((F',F'),(G',G'))\big|
\leq dis(\hat{S}_\delta).
\end{split}
\end{equation}
Since $(F,F'), (G,G') \in S_\delta$ are arbitrary, we obtain $dis(S_\delta) \leq dis(\hat{S}_\delta)$. In the argument for the opposite inequality, we adopt the notation $a \vee b = \max \{a,b\}$ and use the fact that $|a\vee b - c\vee d| \leq |a-c| \vee |b-d|$, for any $a,b,c,d \in \real$. Let $((F_1,F_2),(F'_1,F'_2)) \in \hat{S}_\delta$ and $((G_1,G_2),(G'_1,G'_2)) \in \hat{S}_\delta$. Then,
\begin{equation}
\begin{split}
&\quad \ \big|\dlpc{\mu} ((F_1,F_2),(G_1,G_2)) - \dlpc{\mu'} ((F'_1,F'_2),(G'_1,G'_2))\big| \\
&= \big|\dlp{\mu} (F_1,G_1) \vee \dlp{\mu} (F_2,G_2) - 
\dlp{\mu'} (F'_1,G'_1) \vee \dlp{\mu} (F'_2,G'_2) \big| \\
&\leq \big|\dlp{\mu} (F_1,G_1) - \dlp{\mu'} (F'_1,G'_1)\big| \vee
\big|\dlp{\mu} (F_2,G_2) - \dlp{\mu} (F'_2,G'_2) \big| \\
&\leq dis(S_\delta).
\end{split}
\end{equation}
As the choice of pairs in the correspondence is arbitrary, it follows that $dis(\hat{S}_\delta) \leq dis(S_\delta)$.
\end{proof}

\begin{proof}[\bf Proof of Theorem \ref{T:hcov}]
Let $\delta \in C((d,\mu),(d',\mu'))$ and $h \in \Gamma(\mu,\mu')$, and $\hat{S}_\delta = S_\delta \times S_\delta$ the product correspondence defined in \eqref{E:prodcor}. By Lemmas \ref{L:sdis} and \ref{L:proddis}, the structural distortion of $\hat{S}_\delta$ satisfies
\begin{equation} \label{E:dispair}
dis(\hat{S}_\delta) \leq 4 \Big( \int_{X \times X'} \delta^p (x,y) dh (x,y) \Big)^{1/p}.
\end{equation}
For the functional distortion, Lemma \ref{L:fdis} ensures that
\begin{equation} \label{E:fdispair}
\begin{split}
fdis(\hat{S}_\delta) &\leq 2(D_X + D_{X'}) \int_{X \times X'} \delta (x,y) dh (x,y) \\
&\leq 2(D_X + D_{X'}) \Big(\int_{X \times X'} \delta (x,y) dh (x,y)\Big)^{1/p}.
\end{split}
\end{equation}
From \eqref{E:dispair} and \eqref{E:fdispair}, we obtain
\begin{equation}
\begin{split}
\dgh(\covp{\mu}, \covp{\mu'}) &\leq 
\frac{1}{2} \max\{dis(\hat{S}_\delta), 2fdis(\hat{S}_\delta)\} \\
&\leq \max\{2, 2 (D_X + D_{X'})\} \Big(\int_{X \times X'} \delta (x,y) dh (x,y)\Big)^{1/p}.
\end{split}
\end{equation}
Since the couplings $\delta$ and $h$ are arbitrary, we get $\dgh(\covp{\mu}, \covp{\mu'}) \leq C(X,X') \dks(\mu,\mu')$.
\end{proof}

%--------------------------

\section{Summary and Discussion} \label{S:summary}

\paragraph {Summary.} This paper introduced an approach to statistical analysis of data in metric spaces based on metric observables, which are defined as 1-Lipschitz scalar fields $f \colon X \to \real$. We introduced the concepts of observable mean and observable covariance which form the basis of principal observable analysis. Principal observables were constructed as a hierarchy of metric observables, based on a maximization-of-variance principle, leading to a technique for vectorization, dimension reduction, visualization, and analysis of metric or networked data. Principal observables also provide basis functions for signal analysis in the observable domain. We developed a framework for the formulation of the models and algorithms to calculate principal observables, perform principal observable analysis, and represent signals in the observable domain. We also established the stability of the observable mean and observable covariance for both homogeneous and heterogeneous data.

\paragraph{Perspectives.} Principal observable analysis, as a vectorization technique for metric or networked data, gives an alternative pathway for integration of such data with machine learning tools such as neural networks, a topic that can be further explored. Additional avenues for further investigation also include a more thorough study of representation of signals in the observable domain and construction of basis functions on metric domains that seek to optimize objectives other than observable variance. Other extensions of the methods and tools discussed in this paper include a formulation for (non-compact) Polish $mm$-spaces; that is, $mm$-spaces whose underlying metric spaces are complete and separable.

Whereas principal observables provide an informative summary of the shape of metric data, understanding the shape of the principal observables themselves should lead to further insights into the structure and organization of probability measures on metric spaces. Toward this goal, one can exploit constructs such as merge trees \cite{morozov2013}, Reeb graphs \cite{Biasotti2014}, and other geometric and topological summaries of metric observables $f \colon X \to \real$ such as the decorated  versions of merge trees and Reeb graphs developed in \cite{curry2022,curry2024}.

%--------------------

\bibliographystyle{abbrv}
\bibliography{poabiblio}

%--------------------

\appendix
\section{Appendix}

\subsection{Existence of Principal Observables} \label{A:po}

We first verify the sequential compactness (which is equivalent to compactness) of the metric space $(\cobs{\mu}, \|\cdot\|_\infty)$ of centered principal observables on $(X,d)$ equipped with the metric induced by the $L_\infty$ norm. We need to show that any sequence $f_n \in \cobs{\mu}$, $n \geq 1$, has a convergent subsequence. By Lemma \ref{L:bound}, $\|f_n\|_\infty < D_X$, for every $n \geq 1$, so that the sequence is uniformly bounded. Since each $f_n$ is 1-Lipschitz, the sequence is uniformly equicontinuous. Thus, by the Arzel\`{a}-Ascoli Theorem, $f_n$ admits a uniformly convergent subsequence $f_{n_k}$, $k \geq 1$. Set $f = \lim_{k \to \infty} f_{n_k}$. Since each $f_{n_k}$ is 1-Lipschitz and centered, and the convergence is uniform, the limiting function $f$ is also centered and 1-Lipschitz; that is, $f \in \cobs{\mu}$.

By (ii) of Proposition \ref{P:continuity}, the variance function $f \mapsto \sigma^2(f)$ is continuous on $\cobs{\mu}$. Hence, the compactness of $\cobs{\mu}$ ensures that $\sigma^2$ attains a maximum on $\cobs{\mu}$. This proves that a first principal observable $\phi_1$ exists if $\cobs{\mu}$ contains functions of non-zero variance. Inductively, assume the existence of the first $n-1$ principal observables  $\phi_1, \ldots, \phi_{n-1}$ and define $T_n \colon \lip_\mu^c\to \real^{n-1}$ by $T_n (f) = (\langle f,\phi_1\rangle, \ldots, \langle f,\phi_1\rangle)$, where we adopted the abbreviation $\langle f,g \rangle = \int_X f(x)g(x) d\mu (x)$. Then, $V_n = T_n^{-1} (0) \subseteq \lip_\mu^c$ is the closed subspace of centered Lipschitz functions that are $\mu$-orthogonal to $\phi_1, \ldots, \phi_{n-1}$. Thus, $V_n \cap \cobs{\mu}$ is also compact because it is a closed subset of a compact space. Therefore, if $V_n \cap \cobs{\mu}$ contains functions on non-zero variance, the $n$th principal observable $\phi_n$ exists.

\subsection{Proof of Lemma \ref{L:fdis}} \label{A:fdis}

\begin{proof}[Proof of Lemma \ref{L:fdis}]
We use the same notation as in the proof of Lemma \ref{L:sdis} and use the abbreviations $\Sigma=\Sigma_\mu$ and $\Sigma' = \Sigma_{\mu'}$. For any $(F,F'), (G,G') \in R_\delta$, we have that
\begin{equation} \label{E:covest1}
\begin{split}
|\Sigma(F,G) &- \Sigma'(F',G')| =
|\Sigma (f-Mf,g-Mg)-\Sigma' (f'-M'f',g'-M'g')| \\
&=\Big| \int_{X\times X'} (f(x)-Mf)(g(x)-Mg) - 
(f'(x')-M'f')(g'(x')-M'g') dh(x,x')\Big| \\
&\leq \int_{X\times X'} \big|(f(x)-Mf)(g(x)-Mg) - 
(f'(x')-M'f')(g(x)-Mg) \big| dh (x,x') \\
&\,+ \int_{X\times X'} \big|(f'(x')-M'f')(g(x)-Mg)
- (f'(x')-M'f')(g'(x')-M'g') \big| dh (x,x') \\
&= \int_{X\times X'} \big|(f(x)-Mf) - 
(f'(x')-M'f')\big| \cdot \big|g(x)-Mg\big| dh (x,x') \\
&\,+ \int_{X\times X'} \big|(g(x)-Mg) 
- (g'(x')-M'g')\big| \cdot \big| (f'(x')-M'f') \big| dh (x,x').
\end{split}
\end{equation}
By Lemma \ref{L:bound}, 
\begin{equation}
|g(x)-Mg|\leq D_X \quad \text{and} \quad |f'(x')-M'f'|\leq D_{X'}, 
\end{equation}
$\forall x\in X$ and $\forall x'\in X'$. Thus, \eqref{E:covest1} implies that
\begin{equation} \label{E:covest2}
\begin{split}
|\Sigma(F,G)-\Sigma'(F',G')| &\leq D_X |Mf-M'f'| + 
D_X \int_{X\times X'} |f(x)- f'(x')|dh (x,x') \\
&\,+ D_{X'} |Mg-M'g'| + 
D_{X'} \int_{X\times X'} |g(x)-g'(x')|dh (x,x') \\
&= D_X |Mf-M'f'| + 
D_X \int_{K\times K'} |f(x)- f'(x')|dh (x,x') \\
&\,+ D_{X'} |Mg-M'g'| + 
D_{X'} \int_{K\times K'} |g(x)-g'(x')|dh (x,x') \\
&\leq D_X |Mf-M'f'| + 
D_X \int_{K\times K'} \delta(x,x') dh (x,x') \\
&\,+ D_{X'} |Mg-M'g'| + 
D_{X'} \int_{K\times K'} \delta(x,x') dh (x,x') \\
&= D_X |Mf-M'f'| + 
D_X \int_{X\times X'} \delta(x,x') dh (x,x') \\
&\,+ D_{X'} |Mg-M'g'| + 
D_{X'} \int_{X\times X'} \delta(x,x') dh (x,x'),
\end{split}
\end{equation}
where we have used the fact that $f_K \sqcup f'_{K'}$ and $g_K\sqcup g'_{K'}$ are 1-Lipschitz with respect to $\delta$. From \eqref{E:meandif} and \eqref{E:covest2}, we obtain
\begin{equation}
|\Sigma(F,G)-\Sigma'(F',G')| \leq 2 (D_X + D_{X'}) 
\int_{X\times X'} \delta(x,x') dh (x,x').
\end{equation}
Since $(F,F')$ and $(G,G')$ are arbitrary, we obtain
\begin{equation}
fdis(\hat{S}_\delta) \leq 2 (D_X + D_{X'}) 
\int_{X\times X'} \delta(x,x') dh (x,x'),
\end{equation}
as claimed.
\end{proof}

\end{document}